\documentclass{article}

\usepackage{units,stmaryrd}

\usepackage{amsmath,amssymb,amsthm}

\usepackage{enumerate}
\usepackage{hyperref}

\newtheorem{theorem}{theorem}[section]

\newtheorem{lemma}[theorem]{Lemma}
\newtheorem{cor}[theorem]{cor}
\newtheorem{prop}[theorem]{prop}
\newtheorem{defi}[theorem]{Definition}

\def\w{\ensuremath{\mathsf w}}

\def\x{{\xi}}
\def\y{\zeta}
\def\z{\eta}
\def\u{\delta}

\newcommand{\BNF}{\ensuremath{{\sf BNF}}\xspace}
\newcommand{\B}{\ensuremath{{\mathcal{B}}\xspace}}

\newcommand{\glp}{{\ensuremath{\mathsf{GLP}}}\xspace}

\usepackage{xspace}

\newcommand{\pa}{\ensuremath{{\mathrm{PA}}}\xspace}
\newcommand{\zfc}{\ensuremath{{\mathrm{ZFC}}}\xspace}

\newcommand{\gl}{{\ensuremath{\textup{\textbf{GL}}}}\xspace}

\newcommand{\idel}[1]{{\ensuremath {\mathrm{I}\Delta_{#1}}}\xspace}

\newcommand{\ea}{\ensuremath{{\rm{EA}}}\xspace}

\newcommand{\la}{\langle}
\newcommand{\ra}{\rangle}

\def\le{{\ell}}
\def\ex{e}

\def\fmodels{\xymatri\x{
\ar@{|=}[r]^{<\omega}&
}
}
\def\nmodels{\xymatri\x{
\ar@{|=}[r]^{N}&
}
}
\def\<{\left <}

\def\>{\right >}

\def\cbra{\left \{}
\def\cket{\right \}}

\DeclareSymbolFont{AMSb}{U}{msb}{m}{n}
\DeclareMathSymbol{\N}{\mathbin}{AMSb}{"4E}
\DeclareMathSymbol{\Z}{\mathbin}{AMSb}{"5A}
\DeclareMathSymbol{\R}{\mathbin}{AMSb}{"52}
\DeclareMathSymbol{\Q}{\mathbin}{AMSb}{"51}
\DeclareMathSymbol{\I}{\mathbin}{AMSb}{"49}
\DeclareMathSymbol{\C}{\mathbin}{AMSb}{"43}

\begin{document}

\title{Well-orders in the transfinite Japaridze algebra II:\\
Turing progressions and their well-orders}

\author{David Fern\'andez-Duque,\\
 Department of Mathematics,\\
  Instituto Tecnol\'ogico Aut\'onomo de M\'exico,\\
	david.fernandez@itam.mx\\
 \ \\
 Joost J. Joosten, \\
 Department of Logic, History and Philosophy of Science,\\ University of Barcelona,
 \\ jjoosten@ub.edu}  





\maketitle

\begin{abstract}
  \noindent We study transfinite extensions of Japaridze's provability logic GLP and the well-founded relations that naturally occur within them. Every ordinal induces a partial order over the class of ``words,'' which are iterated consistency statements expressible within GLP. Well-ordered restrictions of these partial orders have been studied previously; in this paper we consider the unrestricted partial orders, which are no longer linear but remain well-founded. These unrestricted partial orders bear important repercussions on modal semantics for GLP and on Turing progressions.
  
Large part of this document has been merged with \cite{FernandezJoosten:2012:WellOrders:Version1} to obtain the document \cite{FernandezJoosten:2012:WellOrders}.  
\\

\end{abstract}

\section{Introduction}

  This paper is a follow-up to \cite{FernandezJoosten:2012:WellOrders, FernandezJoosten:2012:TuringProgressions} and studies the poly-modal provability logics $\glp_\Lambda$ and natural well-founded orders therein. For each ordinal $\Lambda$ one can define a propositional provability logic $\glp_\Lambda$ that has for each $\alpha < \Lambda$ a modal operator $[\alpha]$ corresponding to $\alpha$-provability and a dual operator $\la \alpha \ra$ corresponding to $\alpha$-consistency. By \glp we denote that class-size logic that has a modality for each ordinal.

Worms are iterated consistency statements of the form $\la \alpha_1 \ra \ldots \la \alpha_n \ra \top$. We denote the class of all worms by $S$ and  by $S_\alpha$ denote the class of worms all of whose occurring modalities are at least $\alpha$. $S_\alpha$ can be naturally ordered by defining $A<_\alpha B :\Leftrightarrow \glp\vdash B \to \la \alpha \ra A$. The ordered structures $\la S_\alpha, <_\alpha\ra$ have been extensively studied (\cite{Beklemishev:2005:VeblenInGLP, BeklemishevFernandezJoosten:2012:LinearlyOrderedGLP, FernandezJoosten:2012:WellOrders}) and it is known that, modulo provable equivalence, they define well-orders and constitute alternative ordinal notation systems. 

In particular, for each worm $A\in S_\alpha$ the set $\{B\in S_\alpha \mid B<_\alpha A \}$ is (again, modulo provable equivalence), a well-order with an order-type we shall denote $o_\alpha(A)$. In this paper we see how a small change  in the definition of these sets makes a tremendous difference: the sets $\{B{\in} S \mid B<_\alpha A \}$ are well-founded but exhibit infinite anti-chains for $\alpha >0$. 

Let us denote the supremum of order types of chains in  $\{B{\in} S {\mid} B{<_\alpha} A \}$ by $\Omega_\alpha(A)$. Our main goal is to fully characterize which sequences of ordinals can be attained as $\la \Omega_\xi (A) \ra_{\xi \in {\sf Ord}}$ for some worm $A$ and explore the relation between these sequences, Turing progressions and modal semantics for the closed fragment of $\glp_\Lambda$. 

We give both a local and a global characterization of such sequences. The global characterization is given in terms of so-called \emph{cohyperations} of ordinal functions. Cohyperations are defined as an infinite iterate of particular ordinal functions.

\subsection{Background}

The provability logic of an arithmetic theory $T$ is a modal logic where the $\Box$ modality is interpreted as the formalization of ``provable in $T$". The structural propositional behavior of formalized provability in sound r.e.\ theories is characterized by Solovay's theorem \cite{Solovay:1976} and is known to be the modal logic \gl that we shall introduce below.

It is known that provability logics are very stable in that \emph{any} sound r.e.\ theory that extends some rather weak arithmetic theory as $I\Delta_0 + \exp$ has the same provability logic \gl. And, as a matter of fact, one can also weaken the assumption of recursively enumerable axiomatizability. In particular it is known that \gl is also the logic of provability when interpreting the $\Box$ operator as ``provable with $n$ applications of the omega-rule" or ``provable in $T$ together with all true $\Pi_n$-formulas", etc.

Japaridze (\cite{Japaridze:1988}) introduced a logic $\glp_\omega$ (details follow below) that has for each natural number $n$ a modality $[n]$ where we interpret $[n]$ as ``provable by $n$ applications of the $\omega$-rule''. He showed this logic $\glp_\omega$ to be arithmetically sound and complete for this interpretation. Ignatiev then showed in \cite{Ignatiev:1993:StrongProvabilityPredicates} that this completeness result actually holds for a wide range of arithmetical readings of $[n]$. 

In particular, we still have completeness of $\glp_\omega$ when reading $[n]$ as a natural formalization of ``provable in \ea together with all true $\Pi^0_n$ sentences''. We shall see that under this reading, the logic $\glp_\omega$ is closely related to Turing progressions (also defined below); these are hierarchies of theories of increasing strength introduced by Turing in his doctoral dissertation under Alonzo Church. An historic account of their origin and significance may be found in \cite{feferman}.

Interest in the logic $\glp_\omega$ and related systems recently revived when Beklemishev applied $\glp_\omega$ to perform a $\Pi^0_1$-ordinal analysis for Peano arithmetic (\pa) and related systems (\cite{Beklemishev:2004:ProvabilityAlgebrasAndOrdinals}). 

Moreover, it turned out that $\glp_\omega$ and fragments poses very interesting properties. One can debate over the notion of natural, but arguably it is the first natural example of an axiomatically defined logic that is not Kripke complete but that is complete with respect to its natural class of topological spaces \cite{BekBezhIcard2010}. However, if one restricts oneself to natural ordinal spaces with their corresponding canonical topologies then the question of completeness becomes dependent on set-theoretical assumptions which are themselves independent of \zfc (\cite{Blass:1990:InfinitaryCombinatorics, Beklemishev:2011:OrdinalCompleteness}).

The ordinal analysis that Beklemishev performed for \pa and its kin was actually carried out almost entirely within the closed fragment $\glp^0_\omega$ of $\glp_\omega$, that is, those theorems of $\glp_\omega$ that do not contain propositional variables but rather are built up from $\bot$, $\top$ and the modal and Boolean connectives. Particular terms --so called \emph{worms}-- within this fragment constitute an alternative ordinal notation system for ordinals below $\varepsilon_0$. 

In order to obtain an ordinal notation system based on worms that goes beyond $\varepsilon_0$, Beklemshev considered in \cite{Beklemishev:2005:VeblenInGLP} the logics $\glp_\Lambda$ with $\Lambda > \omega$. These logics are like $\glp_\omega$ with the sole exception that they now contain a modality $[\alpha]$ for each $\alpha < \Lambda$ together with their corresponding axioms and rules. Beklemishev also introduced a class-size logic \glp that contains a modality $[\alpha]$ for each ordinal $\alpha$.

In \cite{Beklemishev:2005:VeblenInGLP}, Beklemishev also established a correspondence between the ordinal notation system based on worms and the more familiar one based on so-called Veblen normal forms. This relation was studied in more detail in \cite{FernandezJoosten:2012:WellOrders} where in particular the authors worked with so-called \emph{hyperations} instead of Veblen functions. 

Hyperations are transfinite iterations of normal ordinal functions which can be seen as a natural refinement of the Veblen functions in particular, and more in general of any Veblen progression of normal ordinal functions. The theory of hyperations is inspired by problems that arose in the study of \glp but is studied and developed on an independent footing by the authors in \cite{FernandezJoosten:2012:Hyperations}.

In the current paper we study natural and important generalizations of the orderings on worms that were studied in both \cite{Beklemishev:2005:VeblenInGLP} and \cite{FernandezJoosten:2012:WellOrders}. Parts of the results presented in this paper have been presented in \cite{FernandezJoosten:2012:TuringProgressions}. In the current paper, these results are presented with more detail. Moreover, we present an important improvement that is based on the techniques of so-called \emph{cohyperations}. In order to sketch an outline of this paper we first need to formally introduce the logics $\glp_\Lambda$, their closed fragments and the worms that dwell therein.

\subsection{\texorpdfstring{The logics $\glp_{\Lambda}$}{The full logics}}
The language of $\glp_{\Lambda}$ is that of propositional modal logic that contains for each $\alpha < \Lambda$ a unary modal operator $[\alpha]$.
In the definition below the $\alpha$ and $\beta$ range over ordinals and the $\psi$ and $\chi$ over formulas in the language of $\glp_{\Lambda}$. 
\begin{defi}
For $\Lambda$ an ordinal, the logic $\glp_{\Lambda}$ is the propositional normal modal logic that has for each $\alpha < \Lambda$ a modality $[\alpha ]$ and is axiomatized by the following schemata:
\[
\begin{array}{ll}
[\alpha] (\chi \to \psi) \to ([\alpha] \chi \to  [\alpha]\psi), & \\
{}[ \alpha ] ([\alpha] \chi \to \chi) \to [\alpha] \chi, &\\
\langle \alpha \rangle \psi \to [\beta] \langle \alpha \rangle \psi & \mbox{for  $\alpha < \beta$,}\\
{}[\alpha] \psi \to [\beta] \psi &\mbox{for  $\alpha \leq \beta$}. 
\end{array}
\]
The rules of inference are Modus Ponens and necessitation for each modality: $\frac{\psi}{[\alpha]\psi}$.
By \glp we denote the class-size logic that has a modality $[\alpha]$ for each ordinal $\alpha$ and all the corresponding axioms and rules.
\end{defi}
It is good to recall that from L\"ob's axiom ${}[ \alpha ] ([\alpha] \chi \to \chi) \to [\alpha] \chi $ one can easily derive transitivity, that is,
\[
[\alpha] \chi \to [\alpha ] [\alpha ] \chi,
\]
and we shall use this freely in our reasoning. The classical G\"odel-L\"ob provability logic \gl is denoted by $\glp_1$.


\subsection{\texorpdfstring{Worms and the closed fragment of \glp}{Worms and the closed fragment}}
A closed formula in the language of \glp is simply a formula without propositional variables. In other words, closed formulas are generated by just $\top$ and the Boolean and modal operators. 

The closed fragment of \glp is just the class of closed formulas provable in \glp and is denoted by $\glp^0$.
Within this closed fragment and the corresponding algebra, there is a particular class of privileged inhabitants/terms which are called \emph{worms}.

\begin{defi}[Worms, $S$, $S_{\alpha}$]
By $S$ we denote the set of \emph{worms} of \glp which is inductively defined as $\top \in S$ and $A\in S \Rightarrow \langle \alpha \rangle A \in S$. Similarly, we inductively define for each ordinal $\alpha$ the set of worms $S_{\alpha}$ where all ordinals are at least $\alpha$ as $\top \in S_{\alpha}$ and $A\in S_{\alpha} \wedge \beta \geq \alpha \Rightarrow \langle \beta \rangle A \in S_\alpha$.
\end{defi}

Both the closed fragment of \glp and the set of worms have been studied  in \cite{Beklemishev:2005:VeblenInGLP} and \cite{BeklemishevFernandezJoosten:2012:LinearlyOrderedGLP}. 
Worms can be conceived as the backbone of $\glp^0$ and obtain their name from the heroic worm-battle, a variant of the Hydra battle (see \cite{Beklemishev:2006}).

We shall identify a worm $A$ in the obvious way with $\iota(A)$, the string of ordinals in the consistency statements that is involved in $A$: $\iota(\top) = \lambda$ and $\iota (\langle \alpha \rangle A) = \alpha {\ast} \iota(A)$. In this paper $\lambda$ will denote the empty string. 

Apart from identifying a worm with its corresponding string of ordinals we shall use any hybrid combination in between at times. For example, we might equally well write $1 0 \omega$, as $\langle 1 \rangle 0 \omega$, or $\langle 1 \rangle \langle 0 \rangle \langle \omega \rangle \top$. Moreover, \emph{par abus de langage} we shall write $\xi \in A$ to denote that the modality $\la \xi \ra$ occurs in the worm $A$.

The following lemma follows easily from the axioms of \glp and shall be used repeatedly without explicit mention in the remainder of this paper.

\begin{lemma}\label{lemma:basicLemma}\ \\
\vspace{-0.4cm}
\begin{enumerate}
\item\label{lemma:basicLemma1}
For a \glp formula $\phi$ and a worm B, if $\beta < \alpha$, then \\
$\glp \vdash (\langle \alpha \rangle \phi \wedge \langle \beta \rangle B) \leftrightarrow \langle \alpha \rangle(\phi \wedge \langle \beta \rangle B)$;

\item\label{lemma:basicLemma2}
If $A\in S_{\alpha+1}$, then $\glp \vdash A \wedge \langle \alpha \rangle B \leftrightarrow A\alpha B$;

\item \label{lemma:basicLemma4}
If $A, B \in S_{\alpha}$ and $\glp \vdash A \leftrightarrow B$, then\\ $\glp \vdash A\alpha C \leftrightarrow B \alpha C$. 
\end{enumerate}
\end{lemma}

\proof
The $\to$ direction of the first item follows from the axiom $\langle \beta \rangle B \to [\alpha] \langle \beta \rangle B$. For the other direction we observe that $\langle \alpha \rangle \langle \beta \rangle B \to \langle \beta \rangle B$ in virtue of axiom $\langle \alpha \rangle \langle \beta \rangle B \to \langle \beta \rangle \langle \beta \rangle B$ and transitivity of $[\beta]$. The other two items follow directly from the first.
\qed


\subsection{Plan of the paper}
After the introduction, in Section \ref{section:OrdinalArithmetic} we will revisit some standard notions from ordinal arithmetic that are needed throughout the rest of the paper.

In Section \ref{section:LinearOrdersOnWorms} we describe the linear orders $<_{\alpha}$ on $S_{\alpha}$ defined as $A<_{\alpha} B :\Leftrightarrow \glp \vdash B \to \langle \alpha \rangle A$. The function $o$ will map a worm to the order type of the set $\{ B\in S \mid B<_0 A\}$. We resume a calculus for computing $o$ as presented in \cite{FernandezJoosten:2012:WellOrders}.
An important ingredient in this calculus is the function $e^{\alpha}$ which is defined as the function that enumerates $o(S_{\alpha})$. The functions $e^\alpha$ can be seen as a transfinite iterate that we call \emph{hyperation}.

Next, in Section \ref{section:OrdersAndGeneralConsiderations} we study the order $<_{\alpha}$ on $S$ in general and not only on $S_{\alpha}$. In this case $<_{\alpha}$ no longer linearly orders $S$ but rather defines a well-founded relation. By $\Omega_{\alpha}(A)$ we will denote the supremum of order-types of linear orders that reside in $\{  B \in S \mid B <_{\alpha}A\}$. We shall see how the study of $\Omega_{\alpha}$'s can be recursively reduced to the study of $o_\xi$'s. Most of the results presented here and in the next section appeared also in \cite{FernandezJoosten:2012:TuringProgressions}.

In Section \ref{section:OmegaSequences} we shall study the sequences $\langle \Omega_{\alpha}(A)\rangle_{\alpha \in {\sf On}}$ for worms $A$ and give a full characterization of these sequences. 

We shall see that these sequences are important for two reasons. Firstly, in Section \ref{section:OrdersAndGeneralConsiderations} we see that they provide us information (lower-bounds, one could say) of what a modal model for the closed fragment of \glp should look like. In Section \ref{section:TuringRevisited} we shall see that the theory $T+A$ for $\glp_\omega$ worms $A$ can exactly be characterized in terms of its Turing progression aproxomations by $\langle \Omega_{\alpha}(A)\rangle_{\alpha \in {\sf On}}$.

The first characterization of these omega sequences that we give is of local nature. In particular, we prove a lemma that determines the nature of the omega-sequences at successor ordinals, and a different lemma for limit ordinals.

In Section \ref{section:FromLocalToGlobal} we take these two lemmata under the loupe and isolate a common feature. To smoothly express this common feature we would need a uniform way to obtain left-inverses to hyperations: which is given by the theory of what we call \emph{cohyperations}.

In Section \ref{section:Cohyperations} we summarize results from the theory of hyperations and cohyperations as presented in \cite{FernandezJoosten:2012:Hyperations}. An important theorem is obtained that characterizes so-called \emph{hyperlogarithms} which are essential in the next section. Hyperations and cohyperations were introduced by the authors in order to give a smooth global presentation of the omega sequences. 

Finally, in Section \ref{section:ComputingOmegaCoordinates} we set the cohyperations at work to obtain a global characterization of the omega sequences.


\subsection{Notation}

We reserve lower-case Greek letters $\alpha, \beta, \gamma, \ldots {\x}\ldots$ for variables ranging over ordinals. Worms will be denoted by upper case latin letters $A, B, C,  \ldots$. The Greek lower-case letters $\phi, \psi, \chi, \ldots$ will denote formulas. However, $\varphi$ shall be reserved for the Veblen enumeration function and variants thereof. Likewise, we reserve $\omega$ to denote the first infinite ordinal.


\section{Turing progressions and modal logic}

The logics \gl and $\glp_\omega$ turn out to be very well suited to talk about Turing progressions. Let us recall the definition of Turing progressions as introduced by Turing in his seminal paper \cite{Turing:1939:TuringProgressions}.

G\"odel's Second Incompleteness Theorem tells us that any sound recursive theory that is strong enough to code syntax will not prove its own consistency. Thus, adding ${\sf Con}(T)$ to such a theory $T$ will yield a strictly stronger theory. Turing took up this idea to consider recursive ordinal progressions of some recursive sound base theory $T$:
\[
\begin{array}{llll}
T_0 &:=& T;  \\
T_{\alpha +1} & :=& T_\alpha + {\sf Con}(T_\alpha); & \\
T_\lambda & := & \bigcup_{\alpha < \lambda} T_\alpha & \mbox{for limit $\lambda$.} 
\end{array}
\]
Poly-modal provability logics turn out to be suitably well equipped to talk about Turing progressions. When talking about closed formulas of \glp we shall often not distinguish a modal formula from its arithmetical interpretation.

Finite Turing progressions are definable in \gl as $T_n$ is provably equivalent to $T+ \Diamond^n_T \top$ where $\Diamond_T \phi$ stand for the arithmetic sentence ${\sf Con}(T+\phi)$.  Transfinite progressions are not expressible in the modal language with just one modal operator. However, using stronger provability predicates provides a way out (see \cite{Beklemishev:2005:Survey}). In particular, the following proposition tells us how to approximate the $\omega$'th Turing progression. For this and the following proposition there are some technical side-conditions on the theory $T$ that shall be specified in Section \ref{section:TuringRevisited}. In the current section, we are mainly interested in seeing the link between Turing progressions and polymodal provability logics.

\begin{prop}\label{theorem:reductionProperty}
$T + \langle n+1 \rangle_T \top$ is a $\Pi_{n+1}$ conservative extension of \\
$T + \{  \langle n \rangle_T ^k \top \mid k \in \omega  \}$.
\end{prop}

More in general we have the following proposition (\cite{Beklemishev:2005:Survey}):

\begin{prop}\label{theorem:generalizedReductionLemma}
For each ordinal $\alpha< \epsilon_0$ there is some $\glp_\omega$-worm $A$ such that $T + A$ is $\Pi_1$ equivalent to $T_\alpha$.
\end{prop}
To get generalizations of this lemma beyond $\epsilon_0$ one should consider more than $\omega$ modalities. Before doing so, in the next section we first provide some more background on the ordinals that we shall need later on in this paper. In the final section, Section \ref{section:TuringRevisited}, we shall see how the omega sequences can be interpreted in terms of Turing progressions.


\section{Ordinal arithmetic}\label{section:OrdinalArithmetic}


In this section we shall briefly state without proof the main properties of ordinals that we need in the remainder of this paper.
For further definitions and detailed proofs, we refer the reader to \cite{Pohlers:2009:PTBook}. Ordinals are canonical representatives for well-orders. The first infinite ordinal is as always denoted by $\omega$. 

Most operations on natural numbers can be extended to ordinal numbers, like addition, multiplication and exponentiation (see \cite{Pohlers:2009:PTBook}). 

\begin{lemma}\label{theorem:BasicPropertiesOrdinalArithmetic}
\

\begin{enumerate}
\item\  $\forall \, \zeta {<} \xi \, \exists ! \eta \ \zeta + \eta = \xi$\\
(We will denote this unique $\eta$ by $-\zeta + \xi$),

\item
$\forall \eta>0 \, \exists \alpha \, \exists !  \beta \ \eta = \alpha + \omega^{\beta}$\\
(We will denote this unique $\beta$ by $\le \eta$),

\item
$\forall \eta>0 \, \exists ! \, \alpha, \beta \ \eta = \omega^{\alpha} + {\beta}$ such that $\beta < \omega^{\alpha}+\beta$.
\end{enumerate}
\end{lemma}

One of the most useful ways to represent ordinals is through their Cantor Normal Forms (CNFs):

\begin{theorem}[Cantor Normal Form Theorem] \ \\
For each ordinal $\alpha$ there are unique ordinals $\alpha_1\geq \ldots \geq \alpha_n$ such that 
\[
\alpha = \omega^{\alpha_1} + \ldots + \omega^{\alpha_n}.
\] 
\end{theorem}
We call a function $f$ \emph{increasing} if $\alpha < \beta$ implies $f(\alpha) < f(\beta)$. An ordinal function is called \emph{continuous} if $\bigcup_{\zeta<\xi}f(\zeta) = f(\xi)$ for all limit ordinals $\xi$. Functions which are both increasing and continuous are called \emph{normal}. 

It is not hard to see that each normal function has an unbounded set of fixpoints. For example the first fixpoint of the function $\varphi_0: x\mapsto \omega^x$ is
\[\sup \{  \omega, \omega^{\omega},\omega^{\omega^{\omega}}, \ldots \}\]
and is denoted $\varepsilon_0$. Clearly for these fixpoints, CNFs give little information as, for example, $\varepsilon_0 = \omega^{\varepsilon_0}$. Therefore, we shall need notations and normal forms that are slightly more informative and which are based on functions that enumerate the fixpoints of normal functions: Veblen Normal Forms (VNFs).

In his seminal paper \cite{Veblen:1908}, Veblen considered for each normal function $f$ its derivative $f'$ that enumerates the fixpoints of $f$. 
%
If $f$ is a normal function, then the image of $f$ --which we shall denote by $F$-- is a closed (under taking supremata) unbounded set. Likewise the function that enumerates a closed unbounded set is continuous.  
For $f$ a normal function, we define $F'$ to be the image of $f'$ and we extend this transfinitely by setting
\[
\begin{array}{llll}
F_{\alpha+1} &:=&(F_{\alpha})';& \\
F_{\lambda} &:=& \displaystyle\bigcap_{\alpha<\lambda} F_{\alpha}& \mbox{ for limit $\lambda$},
\end{array}
\]
then taking $f_\lambda$ to be the function that enumerates $F_{\lambda}$.

By taking $\Phi_0 := \{ \omega^{\alpha} \mid {\alpha}\in {\sf On}\}$ one obtains Veblen's original hierarchy and the $\varphi_{\alpha}$ denote the corresponding enumeration functions of the classes $\Phi_{\alpha}$. 

Beklemishev noted in \cite{Beklemishev:2005:VeblenInGLP} that in the setting of \glp it is desirable to have $1 \notin \Phi_0$. Thus he considered the progression that started with $\Phi_0^B :=  \{ \omega^{1+ \alpha} \mid {\alpha}\in {\sf On}\}$. We denote the corresponding enumeration functions by $\hat \varphi_{\alpha}$. 

In \cite{FernandezJoosten:2012:Hyperations} and in this paper the authors realized that, moreover it is desirable to have $0$ in the initial set, whence we departed from 
\[
{E}_0 = \{0\} \cup \{ \omega^{1+ \alpha} \mid {\alpha}\in {\sf On}\}.
\]
We shall denote the corresponding enumeration functions by $e_{\alpha}$. In general, if $f$ is some normal function, we shall denote by $f_\alpha$ the Veblen progression based on $f_0 =f$. Note that, if $\alpha<\beta$, we have that $f_\beta(\gamma)$ is always a fixpoint of $f_\alpha$, i.e., $f_\beta=f_\alpha\circ f_\beta$.

One readily observes that
\[
\begin{array}{rcccll}
e_{\alpha}(0)&=&0&&&\text{for all $\alpha$;}\\
e_0(1+\beta)&=&\varphi_0(1+\beta)& =& \hat\varphi_0(\beta)&\text{for all $\beta$;}\\
e_{1+\alpha}(1+\beta)&=&\varphi_{1+\alpha}(\beta) &=& \hat\varphi_{1+\alpha}(\beta)&\text{for all $\alpha,\beta$.}\\
\end{array}
\]

Many times, we can write an ordinal $\omega^{\alpha}$ in more than one way as $\varphi_{\xi}(\eta)$. However, if we require that $\eta < \varphi_{\xi}(\eta)$, then both $\xi$ and $\eta$ are uniquely determined. In other words
\[
\forall \alpha \, \exists !\, \eta,\xi\  [\omega^{\alpha} = \varphi_{\xi}(\eta) \ \wedge \ \eta < \varphi_{\xi}(\eta)].
\]
Combining this fact with the CNF Theorem one obtains {\em Veblen Normal Forms} for ordinals.
\begin{theorem}[Veblen Normal Form Theorem]
For all $\alpha$ there exist unique $\alpha_1, \beta_1, \ldots ,\alpha_n,\beta_n$ ($n\geq 0$) such that 
\begin{enumerate}
\item
$\alpha = \varphi_{\alpha_1}(\beta_1) + \ldots + \varphi_{\alpha_n}(\beta_n)$,

\item
$\varphi_{\alpha_i}(\beta_i) \geq \varphi_{\alpha_{i+1}}(\beta_{i+1})$ for $i<n$,

\item
$\beta_i < \varphi_{\alpha_i}(\beta_i)$ for $i\leq n$.
\end{enumerate}
\end{theorem}
\noindent
Note that $\alpha_i\geq\alpha_{i+1}$ does not in general hold in the VNF of $\alpha$. For example,
\[
\omega^{\varepsilon_0+1} + \varepsilon_0 = \varphi_0 (\epsilon_0 +1) + \varphi_1(0) = \varphi_0(\varphi_{\varphi_0(0)}(0)+\varphi_0(0)) + \varphi_{\varphi_0(0)}(0).
\]


\section{Linear orders on the Japaridze algebra}\label{section:LinearOrdersOnWorms}

In this section we shall introduce linear orders on worms, an important theme in our paper.

\subsection{\texorpdfstring{The orderings $<_\alpha$}{The orderings}}

It is known that the class of worms is modulo provable equivalence linearly ordered by consistency strength. That is, two worms are either equivalent or one of the two implies the consistency (0-consistency that is) of the other.

\begin{defi}[$<, <_{\alpha}, o, o_{\alpha}$]
We define a relation $<_{\alpha}$ on $S_{\alpha} \times S_{\alpha}$ by 
\[
A <_{\alpha} B \ :\Leftrightarrow \ \glp \vdash B \to \langle \alpha \rangle A \ \ \ \ \ (\mbox{with $A, B \in S_{\alpha}$}).
\]
For $A \in S_{\alpha}$ we denote by $o_{\alpha}(A)$ the order type of $\{ B \in S_{\alpha} \mid B <_{\alpha} A\}$. More precisely, for $A \in S_{\alpha}$ we define inductively
\[o_\alpha(A)=\sup\cbra o_{\alpha}(B)+1:B\in S_\alpha \ \& \ B<_\alpha A\cket,\]
where $\sup\varnothing=0$.

When $X$ is a set or class we shall denote by $o_{\alpha}(X)$ the image of $X$ under $o_{\alpha}$.
\end{defi}

Instead of $<_0$ and $o_0$ we shall write $<$ and $o$, respectively. In \cite{FernandezJoosten:2012:WellOrders} we described Japaridze algebras and how these algebras are the environments where one most naturally considers our orderings. 

As mentioned before, the relations $<_0$ defines total ordering on $S_0$ modulo provable equivalence. In the following subsection we see how we can choose natural representatives from the equivalences classes by switching to what we call \emph{Beklemishev Normal Forms}.

\subsection{A well-order on Beklemshev Normal Forms}

{\BNF}s are a subclass of $S$ on which $<_0$ does define a linear order as was shown in \cite{BeklemishevFernandezJoosten:2012:LinearlyOrderedGLP,Beklemishev:2005:VeblenInGLP}. In those papers it was also shown that each worm is equivalent to a unique worm in {\BNF} and that this {\BNF} can be found effectively for recursive well-orders. Moreover, if $A\in S_{\alpha}$, then its equivalent in {\BNF} is also in $S_{\alpha}$.

In this section we shall provide a calculus to compute $o_{\alpha}$. Note that it is not at all obvious that $o_\alpha$ is defined everywhere, but this turns out to be the case.

\begin{defi}[Beklemishev Normal Form]\label{defbnf}
A worm $A \in S$ is in {\BNF} (Beklemishev Normal Form) iff
\begin{enumerate}
\item
$A = \lambda$ \ \ \ \ \ \ \ \ \  or,

\item
$A$ is of the form $A_k\alpha \ldots \alpha A_1$ with $\alpha = \min(A)$, $k\geq1$ and $A_i \in S_{\alpha+1}$ such that each $A_i$ is in {\BNF} and moreover 
$A_{i+1} \leq_{\alpha+1}A_{i}$
for each $i<k$.
\end{enumerate}
We shall write \B \ for \BNF and $\B_\alpha$ for $\BNF \cap S_\alpha$.
\end{defi}

\begin{lemma}\label{theorem:WidthOneIsBNF}
Each worm of the form $\alpha^n$, i.e.,  $\overbrace{\langle \alpha \rangle \ldots \langle \alpha \rangle}^{\mbox{$n$ times}} \top$, is in {\BNF}.
\end{lemma}

\proof
This is immediate if we conceive $\alpha^n$ as $\lambda\alpha\lambda\ldots\lambda\alpha\lambda$.
\qed

As announced before, the \BNF{s} form a class of natural representatives for formulas without variables with respect to $o$:

\begin{lemma}\label{theorem:oDefinesAnIsomorphism}
The map $o: \ (\B,<_0) \to (\mathsf{Ord},<)$ defines an isomorphism. 
\end{lemma}


\subsection{\texorpdfstring{A calculus for $o$}{A calculus for o}}


In this subsection we state a calculus for computing $o$  and $o_\alpha$. Proofs and details of the calculus presented here can be found in \cite{FernandezJoosten:2012:WellOrders}.
We first need a syntactical operation that promotes or demotes worms in terms of consistency strength.

\begin{defi}[$\alpha\uparrow$ and $\alpha \downarrow$]
Let $A$ be a worm and $\alpha$ an ordinal. By $\alpha \uparrow A$ we denote the worm that is obtained by simultaneously substituting each $\beta$ that occurs in $A$ by $\alpha + \beta$. 

Likewise, if $A \in S_{\alpha}$ we denote by $\alpha \downarrow A$ the worm that is obtained by replacing simultaneously each $\beta$ in $A$ by $-\alpha+\beta$.
\end{defi}
Note that by Lemma \ref{theorem:BasicPropertiesOrdinalArithmetic}, the operation $\alpha \downarrow$ is well-defined on $S_{\alpha}$. The next lemma enumerates some noteworthy properties of these promoting and demoting operations. 

\begin{lemma}\label{theorem:uparrowProperties}For $\alpha, \beta, \gamma$ ordinals and worms $A,B$ we have:

\begin{enumerate}
\item
$\alpha \uparrow \beta < \alpha \uparrow \gamma \ \Leftrightarrow \ \beta < \gamma$,

\item
$\alpha \uparrow \beta \geq \beta$,

\item\label{item:additivity:theorem:uparrowProperties}
$\alpha \uparrow (\beta \uparrow A) = (\alpha + \beta)\uparrow A$,

\item \label{item:upDownarrowIsIdentity:theorem:uparrowProperties}
$\alpha\downarrow(\beta\uparrow A)=(-\alpha+\beta)\uparrow A$, provided $\alpha\leq \beta$,

\item\label{coadd}
$\alpha\downarrow(\beta\downarrow A)=(\beta+\alpha)\downarrow A$, provided $A\in S_{\beta+\alpha}$,

\item\label{inverse}
$\alpha \uparrow ((\beta+\alpha) \downarrow A) = \beta \downarrow A$ for $A\in S_{\beta+\alpha}$,

\item \label{item:DownDownVersusDownUp:theorem:uparrowProperties}
$(\alpha\downarrow \beta)\downarrow A=\beta\downarrow ( \alpha \uparrow A)$, provided $\alpha\leq \beta$ and $A \in S_{\alpha\downarrow \beta}$,

\item\label{item:SmallerRelationExtendsLargerRelation:theorem:uparrowProperties}
$A <_\alpha B \ \Leftrightarrow \ A<B$ for $A,B \in S_\alpha$,

\item\label{item:comparingDifferentOrders:theorem:uparrowProperties}
$A <_{\xi} B \ \Leftrightarrow \alpha \uparrow A <_{\alpha + \xi}  \alpha \uparrow B$.


\end{enumerate}

\end{lemma}

\proof
The first three items are trivial an proofs of the last two items can be found in \cite{FernandezJoosten:2012:WellOrders}. It is clearly sufficient to prove the other items only for ordinals rather than for worms.
For Item \ref{item:upDownarrowIsIdentity:theorem:uparrowProperties} let $\alpha\leq \beta$ and fix some ordinal $\gamma$. We see that 
\[
\begin{array}{lll}
\alpha + (\alpha \downarrow \beta)\uparrow \gamma & = &\alpha + ((\alpha \downarrow \beta) + \gamma)\\
&  = & (\alpha + (\alpha \downarrow \beta)) + \gamma\\
  & = & \beta + \gamma.\\
\end{array}
\]
Thus, $(\alpha \downarrow \beta)\uparrow \gamma$ is the unique ordinal $\delta$ so that $\alpha + \delta = \beta + \gamma$. In other words, \\
$\alpha\downarrow(\beta\uparrow \gamma)=(-\alpha+\beta)\uparrow \gamma$, provided $\alpha\leq \beta$. 

For Item \ref{coadd} we reason similarly and see for $\gamma \geq (\beta+\alpha)$ that 
\[
\begin{array}{lll}
(\beta + \alpha) + \alpha \downarrow (\beta \downarrow \gamma) & = &\beta +(\alpha + \alpha \downarrow (\beta \downarrow \gamma))\\
& = &\beta + \beta \downarrow \gamma\\
&  = &  \gamma.\\
\end{array}
\]
Thus, $\alpha \downarrow (\beta \downarrow \gamma) = (\beta + \alpha)\downarrow \gamma$ provided $\gamma \geq (\beta+\alpha)$. For Item \ref{inverse}, let $\gamma \geq \beta + \alpha$ whence
\[
\begin{array}{lll}
\beta +  \alpha \uparrow ((\beta + \alpha) \downarrow \gamma) & = &\beta +(\alpha + (\beta + \alpha) \downarrow \gamma)\\
& = &(\beta +\alpha) + (\beta + \alpha) \downarrow \gamma\\
&  = &  \gamma.\\
\end{array}
\]
Thus, $\alpha \uparrow ((\beta + \alpha) \downarrow \gamma)= \beta \downarrow \gamma $ provided $\gamma \geq \beta + \alpha$. For Item \ref{item:DownDownVersusDownUp:theorem:uparrowProperties} let $\alpha\leq \beta$ and $\gamma \geq {\alpha\downarrow \beta}$. 
\[
\begin{array}{lll}
\beta +  (\alpha\downarrow \beta)\downarrow  \gamma & = &(\alpha + \alpha \downarrow \beta ) +  (\alpha\downarrow \beta)\downarrow  \gamma\\
& = &\alpha + (\alpha \downarrow \beta  +  (\alpha\downarrow \beta)\downarrow  \gamma)\\
&  = & \alpha + \gamma\\
&  = & \alpha \uparrow \gamma.\\
\end{array}
\]
Thus, $(\alpha\downarrow \beta)\downarrow \gamma=\beta\downarrow ( \alpha \uparrow \gamma)$, provided $\alpha\leq \beta$ and $\gamma \geq {\alpha\downarrow \beta}$.

\qed

Note that Items \ref{item:additivity:theorem:uparrowProperties} --- \ref{coadd} can be seen as some associative laws if we formulate them as $\alpha \uparrow (\beta \uparrow \gamma) = (\alpha \uparrow \beta )\uparrow \gamma$, as $\alpha \downarrow (\beta \uparrow \gamma) = (\alpha \downarrow \beta) \uparrow \gamma$ for $\alpha \leq \beta$, and as $\alpha \downarrow (\beta \downarrow \gamma) = (\beta \uparrow \alpha) \downarrow \gamma$ provided $\gamma \geq \beta + \alpha$ respectively. 

However, we do not have a general expression expressing some for of associativity for $\alpha \uparrow (\beta \downarrow \gamma)$ when $\gamma \geq \beta$ and $\alpha$ and $\beta$ entirely unrelated. When $\alpha$ and $\beta$ bear some relation partial results can be obtained such as Item \ref{inverse} to the effect that $\alpha \uparrow ((\beta\uparrow\alpha) \downarrow \gamma) = \beta \downarrow \gamma$ for $\gamma \geq {\beta+\alpha}$. Likewise, one can show $(\alpha \downarrow \beta) \uparrow (\beta \downarrow \gamma) = \alpha \downarrow \gamma$ for $\alpha \leq \beta \leq \gamma$. 

Note that by our results we do have some form of associativity for $(\cdot \circ \cdot) \circ' \cdot$ for all combinations of $\circ, \circ' \in \{  \uparrow, \downarrow \}$. It is unclear whether all equalities in the language $\{ \uparrow, \downarrow \}$ can be finitely axiomatized. We conjecture the corresponding first-order theory to be decidable.\\
\medskip

In \cite{FernandezJoosten:2012:WellOrders} it is proven that $\alpha{\uparrow}$ is a well-behaved map with nice properties. In particular, $\alpha{\uparrow}$ can also be viewed as an isomorphism:

\begin{lemma}\label{theorem:IsomorphicWormFragments}
The map $\alpha {\uparrow}$ is an isomorphism between $(S, <)$ and $(S_\alpha,<_{\alpha})$.
\end{lemma}

In \cite{FernandezJoosten:2012:WellOrders} we introduced the functions $e^\alpha$ that we call \emph{hyperexponentionals}. 

\begin{defi}\label{theorem:recursiveSchemeForHyperexponentials}
For ordinals $\alpha$ and $\beta$, the values $e^\alpha(\beta)$ are determined by the following recursion.
\begin{enumerate}
\item
$e^\alpha 0 =0$ for all $\alpha \in {\sf Ord}$;
\item
$e^1 = e$ where $e$ enumerates the set $\{ 0\} \cup \{ \omega^{1+\alpha} \mid \alpha \in {\sf Ord}\}$;
\item
$e^{\alpha + \beta} = e^{\alpha} e^{\beta}$;
\item
$e^\alpha (\lambda) = \cup_{\beta<\lambda} e^\alpha (\beta)$ for additively indecomposable limit ordinals $\lambda$;

\item
$e^{\lambda} (\beta +1) = \cup_{\lambda' < \lambda} e^{\lambda'} (e^{\lambda}(\beta) +1)$ for $\lambda$ an additively indecomposable limit ordinal.

\end{enumerate}
\end{defi}

By an easy induction one can check that each $e^\alpha$ is a normal function.
Based on these hyperexponential functions $e^\alpha$ we can formulate an elegant calculus to compute the values of $o_\alpha(A)$:

\begin{theorem}\label{theorem:OrderTypeCalculus}
\ \\
\vspace{-0.4 cm}
\begin{enumerate}
\item
$o(0^n)=n$;

\item
If ${A} = A_n 0 \hdots  A_1 \in \B$ and $A_1\in \B_1$ is not empty, then\\
 $o({A})= \omega^{o(1\downarrow A_1)}+\hdots+\omega^{o(1\downarrow A_n)}$, where\\
 for $n=1$ we denote by $A_n 0 \hdots  A_1$ simply $A_1$;

\item\label{exo}
$o({\x}\uparrow {A}) = \ex^{\x} o({A})$,

\item\label{down}
$o_{\x} (A) = o(\x \downarrow A)$ for $A \in S_{\x}$.

\end{enumerate}
\end{theorem}

Note that the last item of this theorem is not needed to compute $o$. It merely tells us how to reduce $o_\alpha$ to $o$.
The $e^\alpha$ functions can be related to the more familiar Veblen progressions.

\begin{lemma}\label{theorem:OrderOFWormsOfOmegaPowersAreFixedPoints}
$e^{\omega^{\alpha}} = e_\alpha$.
\end{lemma}

Moreover, we note that Lemma \ref{theorem:OrderOFWormsOfOmegaPowersAreFixedPoints} together with Theorem \ref{theorem:recursiveSchemeForHyperexponentials}.3 yields a reduction of computing $e^\alpha$ to the better known Veblen-like functions $e_\alpha$. For if $\alpha = \omega^{\alpha_1}+\ldots +\omega^{\alpha_n}$, then 
\[
e^\alpha = e_{\alpha_1} \circ \ldots \circ e_{\alpha_n}.
\]


\section{Well-founded orders on worms}\label{section:OrdersAndGeneralConsiderations}


In this section we consider the ordering $<_\alpha$ on the full $S\times S$ rather than on $S_\alpha\times S_\alpha$. We shall see that the resulting order is still well-founded but no longer total. Most of the results presented here and in the next section --with the exception of Subsection \ref{section:antiChains}-- were also presented in a similar form in \cite{FernandezJoosten:2012:TuringProgressions}.


\subsection{Well-founded orders}


In Section \ref{section:LinearOrdersOnWorms} we presented the well-orders $<_{\alpha}$ on $S_{\alpha}$. We can also consider the ordering $<_{\alpha}$ on the full class $S$. As we shall see, $<_\alpha$ is no longer linear on $S$. However, as we shall see in Corollary \ref{theorem:UnrestrictedOrderIsWellFounded}, it is still well-founded. Anticipating this, we can define $\Omega_\alpha(A)$, the generalized $<_\alpha$ order-type of a worm $A$.

\begin{defi}\label{definition:OmegaOrder}
Given an ordinal $\x$ and a worm $A$, we define a new ordinal $\Omega_{\x}({A})$ inductively on $<_{\x}$ by
\[
\Omega_\x(A)=\sup_{B<_\x A}(\Omega_\xi(B) +1).
\]
\end{defi}

\noindent
With this, we can assign to each worm $A$ a sequence of order-types.

We will use the notation $\vec{\Omega}({A})$ for the sequence $\la\Omega_\xi(A)\ra_{\xi\in\mathsf{On}}$; that is,
\[
\vec{\Omega}({A})\ \ :=\ \ (\Omega_0({A}), \ \Omega_1({A}),\ \ldots , \Omega_{\omega}({A}),\ \Omega_{\omega+1}({A})\ldots )\ .
\]
We shall refer to these sequences as \emph{Omega-sequences}.


\subsection{Omega-sequences and modal semantics}


Each worm $A$ is known to be consistent with \glp, hence should be satisfied in an exact 
model for its closed fragment, if it exists; that is, a model on which only the theorems of $\glp^0$ are valid.

Suppose $\mathcal{M}$ were such a model. Each modality $\langle \xi \rangle$ will be represented in $\mathcal{M}$ by some relation $\prec_\xi$ in that
\[
\mathcal{M}, \w \Vdash \langle \xi \rangle \phi \ \Leftrightarrow \ \exists \w' \, (\w' \prec_\xi \w \wedge \mathcal{M},\w'\Vdash \phi).
\]
As $[\xi]$ satisfies L\"ob's axiom, we know that each $\prec_\xi$ is transitive and well-founded. 
Consequently, we can assign to each world ${\w}$ a sequence of ordinals 
\[
\vec{\w} \ := \ ({\w}_0, {\w}_1, \ldots {\w}_{\omega}, {\w}_{\omega+1}\ldots),
\] 
where ${\w}_\zeta$ corresponds to the supremum of order-types of $<_\zeta$-chains below ${\w}$. If $\mathcal{M},\w\Vdash A$, then necessarily $\w_\xi \geq \Omega_\xi(A)$ for each $\xi$. A systematic study of $\vec{\Omega}(A)$ will thus also reveal information about models for $\glp^0$.

No such models were known, but in \cite{FernandezJoosten:2012:ModelsOfGLP} the authors define a universal class-size model for $\glp^0$. The worlds in that model closely reflect the $\Omega_\xi(A)$ sequences as defined here. In particular, it turns out that the necessary condition that if $\mathcal{M},\w\Vdash A$, then  $\w_\xi \geq \Omega_\xi(A)$ for each $\xi$ is actually also sufficient.

In Section \ref{section:OmegaSequences} we shall characterize the sequences $\Omega_\xi(A)$ for given $\xi$ and $A$. In the next subsection we shall see how questions about $\Omega_\xi$ can be recursively reduced to questions about $o_\zeta$. 

\subsection{\texorpdfstring{Reducing $\Omega_\xi$ to $o_\zeta$}{Reducing to o}}

In Lemma \ref{lemma:reducingGeneralOrderToSpecialOrder} below we shall see how questions about $\Omega_\xi$ can be recursively reduced to questions about $o_\zeta$.
For this reduction we need the syntactical definitions of \emph{head} and \emph{remainder}.

\begin{defi}
Let $A$ be a worm. By $h_{\x}(A)$ we denote the \emph{${\x}$-head} of ${A}$. Recursively: $h_{\x}(\lambda)=\lambda$, $h_{\x}(\zeta{\ast}{A})= \zeta{\ast} h_{\x}({A})$ if $\zeta\geq {\x}$ and $h_{\x}(\zeta{\ast}{A})= \lambda$ if $\zeta < {\x}$.

Likewise, by $r_{\x}({A})$ we denote the \emph{${\x}$-remainder} of $A$: $r_{\x}(\lambda)=\lambda$, $r_{\x}(\zeta{\ast}{A})= r_{\x}({A})$ if $\zeta\geq {\x}$ and $r_{\x}(\zeta{\ast}{A})= \zeta{\ast}{A}$ if $\zeta < {\x}$.
\end{defi}

In words, $h_{\x}({A})$ corresponds to the largest initial part (reading from left to right) of ${A}$ such that all symbols in $h_{\x}({A})$ are at least ${\x}$ and $r_{\x}({A})$ is that part of ${A}$ that remains when removing its ${\x}$-head. We thus have ${A} = h_{\x}({A}) {\ast} r_{\x}({A})$ for all ${\x}$ and ${A}$. 

Observe that
\begin{equation}\label{heads}
\glp\vdash h_{\x}({A}) {\ast} r_{\x}({A})\leftrightarrow h_{\x}({A}) \wedge r_{\x}({A}),
\end{equation}
as the first symbol of $r_{\x}({A})$ is less than ${\x}$ and $h_{\x}({A}) \in S_{\x}$ (see Lemma \ref{lemma:basicLemma}). Moreover, for each ${\x}$ and each ${A}$ we have that $h_{\x}({A})$ is in normal form whenever $A$ is:

\begin{lemma}\label{theorem:TakingHeadsPreservesNormalForms}
If ${A} \in $ {\BNF}, then also $h_{\zeta}({A}) \in $ {\BNF} and $r_\zeta(A) \in $ {\BNF}.
\end{lemma}

\proof
We prove here the $h_\zeta(A)$ case. For ${A} = \lambda$ this is clear. Thus, let the symbols in $A$ be enumerated in increasing order by ${\x}_0, \ldots,{\x}_n$. By an easy induction on $n$ we see that each $h_{{\x}_i}({A}) \in$ {\BNF}. 
If ${\x}_n>\zeta\notin A$, then $h_\zeta({A})= h_{\min\{ {\x}_i {\mid} {\x}_i >\zeta \}}({A})$. If $\zeta>{\x}_n$, then $h_\zeta({A})={\lambda}$ which is in {\BNF}.
\qed

\begin{lemma}\label{lemma:reducingGeneralOrderToSpecialOrder}
Let $A$ and $B$ be worms. We have that
\[(A \to \<  {\x} \> B) \Leftrightarrow [ (h_{\x}(A) \to \< {\x}\> h_{\x}(B)) \wedge (A \to r_{\x}(B))].
\]
\end{lemma}

\proof
``$\Rightarrow$" By \ref{heads}, $B \leftrightarrow h_{\x}(B) \wedge r_{\x}(B)$ whence $A \to r_{\x}(B)$ as
\[
\begin{array}{llll}
A &\to& \<  {\x} \> B& \\
 & \to& \<  {\x} \> (h_{\x}(B) \wedge r_{\x}(B)) & \mbox{by Lemma \ref{lemma:basicLemma}.\ref{lemma:basicLemma2}}\\
 & \to &  r_{\x}(B) \wedge \<  {\x} \> h_{\x}(B) & \\
  & \to &  r_{\x}(B).
\end{array}
\]
Likewise $A \leftrightarrow h_{\x}(A) \wedge r_{\x}(A)$. As $h_{\x}(A), h_{\x}(B) \in S_{\x}$ we know that either
\begin{itemize}
\item $h_{\x}(A) \leftrightarrow h_{\x}(B)$,
\item $h_{\x}(B) \to \<\x\> h_{\x}(A)$ or,
\item $h_{\x}(A)\to \< \x \> h_{\x}(B)$.
\end{itemize}
By assumption $A \to \<  \x \> B$ whence $A \to \<  \x \> h_{\x}(B) \wedge r_{\x}(B)$. 

Suppose now $h_{\x}(A) \leftrightarrow h_{\x}(B)$. Then,
\[h_{\x}(A) \wedge r_{\x}(A) \to \<\x\> h_{\x}(A) \wedge r_{\x}(A)\]
whence also
\[h_{\x}(A) \wedge r_{\x}(A) \to \<\x\> (h_{\x}(A) \wedge r_{\x}(A)).\]
The latter is equivalent to $A \to \< \x\> A$ which contradicts the  irreflexivity of $<_{\x}$.

By a similar argument, the assumption that $h_{\x}(B) \to \<\x\> h_{\x}(A)$ contradicts the irreflexivity of $<_{\x}$ and we conclude that $h_{\x}(A)\to \< \x \> h_{\x}(B)$.

``$\Leftarrow$" This is the easier direction.
\[
\begin{array}{lll}
A & \leftrightarrow & h_{\x}(A) \wedge r_{\x}(A)\\
\ &\to & \< \x \> h_{\x}(B) \wedge r_{\x}(B)\\
\ &\to & \<\x\> ( h_{\x}(B) \wedge r_{\x}(B))\\
\ &\to & \< \x\> B.
\end{array}
\]

\qed
In the right hand side of this Lemma \ref{lemma:reducingGeneralOrderToSpecialOrder} we see that the first conjunct $(h_{\x}(A) \to \< {\x}\> h_{\x}(B))$ is only referring to worms in $S_\x$ and their $<_\x$ relations. The worm $r_{\x}(B)$ starts with a modality strictly less than $\x$ and thus  the second conjunct $(A \to r_{\x}(B))$ of the lemma can be settled by calling recursively to the lemma once more. Thus, indeed,  Lemma \ref{lemma:reducingGeneralOrderToSpecialOrder} recursively reduces the general $<_{\x}$ question between worms, to the $<_{\x}$ questions between worms in $S_{\x}$. 

\begin{cor}\label{theorem:UnrestrictedOrderIsWellFounded}
The relation $<_\alpha$ on $S \times S$ is well-founded.
\end{cor}

\begin{proof}
Any $<_\alpha$ descending chain $\ldots <_\alpha A_2 <_\alpha A_1 <_\alpha A_0$ in $S$ yields a corresponding chain $\ldots <_\alpha h_\alpha(A_2) <_\alpha h_\alpha(A_1) <_\alpha h_\alpha(A_0)$ in $S_\alpha$ by Lemma \ref{lemma:reducingGeneralOrderToSpecialOrder}. 
\end{proof}

Now that we have established the well-foundedness of $<_\alpha$ on $S \times S$ we see that Definition \ref{definition:OmegaOrder} is indeed well-defined. Moreover, we may now perform induction on $\Omega_\xi$.

\subsection{Antichains}\label{section:antiChains}

It is easy to see that $<_{\x}$ is not tree-like; for example, we see that both $011<_1 10111 <_1 1111$ and $011<_1 11011 <_1 1111$ while $10111$ and $11011$ are $<_1$ incomparable.

A set of elements $\{  A_i\}_{i<\zeta}$ is called an \emph{anti-chain} for $<\alpha$ if for all $i\neq j$ we have that $A_i$ and $A_j$ are $<_\alpha$-incomparable. That is, $A_i \not \leq A_j$ and $A_j \not < A_i$. An ordered set $\la X, \prec\ra$ is called a \emph{well-quasi order} if $X$ contains no infinite anti-chains for $\prec$.

For $\alpha >0$, we have that $<_\alpha$ does not define a well-quasi-ordering on $S$. For example, all elements $\{ \langle \beta \rangle \top \mid \beta < \alpha \}$ are mutually $<_\alpha$ incomparable yielding us an infinite anti-chain. A natural questions to study for the $<_\alpha$ orderings on $S\times S$ concerns the $<_0$ length of anti-chains. So, given a worm $A$, we can consider sets $X_i = \{  B \mid B <_\alpha A \}$ so that all elements in $X_i$ are mutually $<_\alpha $-incomparable. The question arises, what is $\sup_i{\sf ot}(X_i, <_0)$? Or more in general, what is $\sup_i{\sf ot}(X_i, <_\beta)$ for $\beta<\alpha$?

More generally, we can ask for the supremum of order-types of the $<_\alpha$ anti-chains that lie in between two $<_\alpha$ comparable elements. For example, the set $\{ 101, 10101, 1010101, \ldots\}$ defines an $<_1$ anti-chain of $<_0$ order-type $\omega$ between $1$ and $11$.

It is important to somehow bound where the anti-chain can reside, if not we get anti-chains of arbitrary length. For example,  $\{10\alpha \mid \alpha \in {\sf On}\setminus \{ 0\} \}$ defines an anti-chain w.r.t.\ the $<_1$ order that is unbounded in the $<_0$ order.

Currently it is not clear how to give a sensible arithmetical interpretation of anti-chains in the Japaridze algebra (if possible at all). We shall briefly outline here that anti-chains do not yield sequences of mutually non-interpretable sentences and refer the reader to for example \cite{Joosten:2004:InterpretabilityFormalized} or \cite{visser:1997:OverviewIL} for details. Basically this is due to the effect that interpretability and $\geq_0$ coincide on the class of worms. Let us first fix some notation. By $A\rhd B$ we denote that $T+A$ interprets $T+B$. That is, there is some structure preserving translation $j$ that maps symbols of $T$ to formulas of arithmetic which transforms every $T+B$ theorem into a $T+A$ theorem:
\[
A\rhd B \ \ := \ \ \exists j \, \forall \phi\  (\Box_{T+B}\phi \to \Box_{T+A}\phi^j).
\]

\begin{lemma}\label{bla}
For any pair of worms $A$ and $B$ and consistent base theory $T$ w.r.t.\ which \glp is sound, we have 
\[
A \rhd B \ \ \Longleftrightarrow \ \ A\geq B.
\]
\end{lemma}

\proof
The case that $A=B$ is trivial so we may assume them different.
If $T \vdash A \to \Diamond B$, then we can formalize (already within in rather weak theories like $\idel{0} + \exp$) the Henkin construction so that $\Diamond B$ defines an internal model of $T+B$. This model in its turn defines the translation $j$, so that we obtain $A\rhd B$. 

Suppose now $A\rhd B$ but $\neg (A \geq B)$. By linearity of $<_0$ we would get $B>A$, whence $T\vdash B \to \Diamond A$. Now, using the identity interpretation, we see that $B\rhd \Diamond A$. By transitivity of interpretability, we get $A\rhd \Diamond A$ which contradicts Feferman's generalization of G\"odel's Second Incompleteness Theorem to the effect that no consistent theory can interpret its own consistency.
\qed

One could easily define a generalized notion of interpretability
\[
A\rhd_n B \ \ := \ \ \exists j \, \forall \phi \ ([n]_{T+B}\phi \to [n]_{T+A}\phi^j)
\]
but it is not clear whether $\rhd_n$ would coincide with $\geq_n$ on the class of all worms.

\section{Omega sequences}\label{section:OmegaSequences}

In this section we give a full characterization of the sequences $\vec{\Omega}({A})$; that is, we shall determine for given $A$ each of the values $\Omega_\xi(A)$ and classify at what coordinates $\xi$ the $\vec{\Omega}({A})$ sequence chan\-ges value. 

\subsection{Basic properties of omega sequences}
Clearly, $\vec{\Omega}({A})$ defines a weakly decreasing sequence of ordinals.

\begin{lemma}\label{theorem:WeaklyDecreasingSequences}
For ${\x}<\zeta$ we have that $\Omega_{\x}({A})\geq \Omega_\zeta({A})$.
\end{lemma}

\proof
By induction on $\Omega_\x (A)$ we see that
\[
\begin{array}{lll}
\Omega_\x (A) & := & \sup \{ \Omega_\x (B) +1 \mid B <_\x A \}\\
 & \geq_{\sf IH} & \sup \{ \Omega_\zeta (B) +1 \mid B <_\x A \}\\
  & \geq & \sup \{ \Omega_\zeta (B) +1 \mid B <_\zeta A \}\\
  &= & \Omega_\zeta (A).
\end{array}
\]

Note that we have the last inequality since $\{ B \mid B <_\zeta A \} \subseteq \{ B \mid B <_\x A \}$ (because for ${\x}<\zeta$ we have  $\vdash {A} \to \langle \zeta \rangle {B}$ implies $\vdash  {A} \to \langle {\x} \rangle {B}$). 
\qed
In particular, since the omega sequences are weakly decreasing on the ordinals, we have that $\{ \Omega_\xi(A) \mid \xi \in {\sf Ord} \}$ is a finite set for any worm $A$.

\begin{lemma}\label{theorem:OmegaReducesToO}
$\Omega_{\x}(A) = o_{\x}h_{\x}(A)$
\end{lemma}

\proof
We first see that 
\[
\{ C \in S_\x \mid C <_\x h_\x A \} = \{ h_\x (B) \mid B<_\x A \}. \ \ \ (*)
\]
The inclusion $\subseteq$ is immediate since $h_\x (C) = C$ for $C\in S_\x$. The other direction follows directly from Lemma \ref{lemma:reducingGeneralOrderToSpecialOrder} since $B<_\x A$ implies $h_x(B) <_\x h_\x A$ and clearly $h_x(B)\in S_\x$. Now that we have this equality we proceed by induction and obtain
\[
\begin{array}{lll}
\Omega_\x (A) & := & \sup \{ \Omega_\x (B) +1 \mid B <_\x A \}\\
 & =_{\sf IH} & \sup \{ o_\x h_\x (B) +1 \mid B <_\x A \}\\
 & =_{\mbox{by } (*)} & \sup \{ o_\x (C) +1 \mid C\in S_\x \wedge C <_\x h_\x(A) \}\\
  &= & o_\x h_\x (A).
\end{array}
\]
\qed
As an immediate corollary to this lemma we see that $\Omega_\x(A)$ is actually equal to the supremum of order-types of $<_\x$ chains below $A$. 
\begin{cor}\label{theorem:MaximalCoordinateOfSequences}
For each worm $A\neq \lambda$, there is a maximal $\x$ so that $\Omega_{\x}({A})\neq 0$. In particular we have $\x ={\sf First}({A})$, where ${\sf First}({A})$ is the left-most element of ${A}$.
\end{cor}

\proof
For ${A} \in S$, we denote by ${\sf First}({A})$ the first element of ${A}$, that is, ${\sf First}(\lambda)= \lambda$, and ${\sf First}(\x{\ast} {B})= \x$. 
Clearly, if $A\neq \lambda$ then $h_{\sf First({A})} ({A}) \neq \lambda$ whence by Lemma \ref{theorem:OmegaReducesToO},
$\Omega_{{\sf First}({A})}({A})\neq 0$.
On the other hand, for $\x> {\sf First}({A})$, clearly $h_{\x}({A}) = \lambda$ whence $\Omega_{\x}({A}) =0$.
\qed

It is good to have reduced $\Omega_{\x}({A})$ to $o_{\x}({A})$ as in Section \ref{section:LinearOrdersOnWorms} we provided a full calculus for the latter (Lemma \ref{theorem:OrderTypeCalculus}).

Lemma \ref{theorem:WeaklyDecreasingSequences} and Corollary \ref{theorem:MaximalCoordinateOfSequences} are first simple observations on $\vec{\Omega}({A})$ sequences. In the remainder of this section we shall provide a full characterization of them.


\subsection{Successor coordinates}


First let us compute $\Omega_{\x+1}(A)$ in terms of $\Omega_\x(A)$. Recall that $\le \alpha$ denotes the unique $\beta$ such that $\alpha = \alpha' + \omega^{\beta}$ for $\alpha >0$. For convenience we define $\le 0 = 0$. The following lemma will be useful:

\begin{lemma}\label{lemma:successorLemma}
Given an ordinal ${\x}$ and a worm ${A}$,
\[o_{{\x}+1}h_{{\x}+1}({A}) = \le o_{\x} h_{\x}({A}).\]
\end{lemma}

\proof
We write $h_{\x}({A})$ as $A_0 {\x} \ldots {\x} A_n$. Clearly, $h_{{\x}+1}({A}) = A_0$. We shall now see that $\le o_{\x}h_{\x}({A})=o_{{\x}+1} (A_0)$.

To this end, we observe that 
\[
\begin{array}{lll}
\displaystyle o_{\x}h_{\x}({A}) & = & o_{\x}(A_0 \x \ldots \x A_n) \\\\
\ & = &
o\Big(({\x}{\downarrow}A_0) 0 \ldots 0 ({\x}{\downarrow}A_n)\Big)\\\\
\ & = & \omega^{o_1({\x}{\downarrow} A_n)} + \ldots +\omega^{o_1({\x}{\downarrow} A_0)}\\\\
\ & = & \omega^{o_{{\x}+1}(A_n)} + \ldots +\omega^{o_{{\x}+1}(A_0)}
\end{array}
\]
Consequently $\le o_{\x} h_{\x}({A})=o_{{\x}+1}(A_0)$, as desired.
\qed

Now we are ready to describe the relation between successor coordinates of the $\vec{\Omega}({A})$ sequence.
\begin{theorem}\label{theorem:SuccessorRelations}
$\Omega_{{\x}+1}({A}) = \le\Omega_{\x}({A})$
\end{theorem}

\proof
\[
\begin{array}{llll}
\Omega_{{\x}+1}({A}) & = & o_{{\x}+1}h_{{\x}+1}({A})& \mbox{by Lemma \ref{lemma:successorLemma}}\\
\ & = & \le o_{\x} h_{\x}({A})& \ \\
\ & = & \le\Omega_{\x}({A})& \mbox{by Lemma \ref{lemma:reducingGeneralOrderToSpecialOrder}}. \\
\end{array}
\]
\qed

Theorem \ref{theorem:SuccessorRelations} tells us what the relation between successor coordinates of $\vec{\Omega}({A})$ is. We may also infer from it when successor coordinates are different; if $\Omega_{\x}({A})$ is a fixed point of $\zeta \mapsto \omega^\zeta$ then $\Omega_{\x}({A})= \Omega_{{\x}+1}({A})$. 

\subsection{Equal coordinates}

Theorem \ref{theorem:EqualityCoordinates} below gives us a characterization of when different coordinates attain different or equal values. Before we can state and prove this theorem we first need some notation and background reasoning on CNFs.

For $\alpha \in {\sf On}$ we define $N_\alpha$ and the syntactic operation ${\sf CNF}(\alpha) := \sum_{i=1}^{N_\alpha}\omega^{{\x}_i}$ to be the unique CNF expression of $\alpha$. Next, we define for an ordinal $\alpha$ the set of its \emph{Cantor Normal Form Approximations} as the set of partial sums of ${\sf CNF}(\alpha)$, that is, if
\[{\sf CNF}(\alpha) = \sum_{i=1}^{N_\alpha}\omega^{{\x}_i},\]
then
\[
{\sf CNA}(\alpha) \ := \ \cbra  \sum_{i=1}^k \omega^{{\x}_i} : 0\leq k \leq N_\alpha\cket.
\]
We also define the \emph{Cantor Normal Form Projection} of some ordinal $\y$ on another ordinal $\x$ as follows:
\[
{\sf CNP}(\y,\x) \ \ := \ \ \max \{ \x' {\in} {\sf CNA}(\x) \mid \x' \leq \y\}.
\]
Note that $0\in {\sf CNA(\xi)}$ and that ${\sf CNP}(\y,\x)$ is defined for all $\y,\x \in {\sf On}$.

For $\alpha,\beta,\gamma \in {\sf On}$ we define 
\[
\alpha\sim_\gamma \beta \ \ : \Leftrightarrow \ \ 
{\sf CNP}(\alpha, \gamma) = {\sf CNP}(\beta, \gamma) .
\]
In words, $\alpha\sim_\gamma \beta$ whenever there is no partial sum of the CNF of $\gamma$ that falls in between $\alpha$ and $\beta$ (also the case that both $\alpha$ and $\beta$ are non-equal partial sums is excluded).

The just-defined notions of ${\sf CNA}({\x})$, ${\sf CNP}(\y,\x)$ and $\alpha\sim_\gamma \beta$ are needed to characterize the ${\x}{\downarrow}\zeta$ operation.

\begin{lemma}\label{theorem:DownArrowCharacterization}
Let $\zeta, \xi$ and $\eta$ be ordinals.
\begin{enumerate}
\item
$\displaystyle\forall \zeta{\leq} {\x}  \ \ \  \y{\downarrow}\x  =  
{\sf CNP}(\zeta,\xi){\downarrow}\x$;\\
\item
$\displaystyle\forall \y{\leq} \x \,  \exists ! \z {\in} {\sf CNA}(\x) \  \ \y{\downarrow}\x = \z{\downarrow}\x$;\\
\item
For $\x,\y \leq \z$, we have ${\x}{\downarrow}\z = \y{\downarrow}\z \ \Leftrightarrow \ \x\sim_\z \y $.

\end{enumerate}
\end{lemma}

\proof
1.\ We consider $\y\leq \x$. Now let $\z=\max\{\z' \in {\sf CNA}(\x)\mid \z'\leq \y\}={\sf CNP}(\zeta,\xi)$.
The claim is that $\y{\downarrow}\x=\z{\downarrow}\x$. Let
\[
{\sf CNF}(\x)= \sum_{i=1}^{N_{\x}}\omega^{\x_i}.
\]
As $\z= \sum_{i=1}^{k}\omega^{\x_i}$ for some $k\leq N_{\x}$, we see that 
\[
\z{\downarrow}\x= \sum_{i=k+1}^{N_\x}\omega^{\x_i}
\]
for $k<N_{\x}$ and $\z{\downarrow}\x=0$ for $k=N_{\x}$. We now claim that $\y + (\z{\downarrow}\x)=\x$ so that $\y{\downarrow}\x = \z{\downarrow}\x$ follows from the fact that
\[
\forall\, \y {<} \x\, \exists ! \u \ \y+\u = \x.
\]
We may assume $\y>\z$ otherwise $\y + (\z{\downarrow}\x)=\x$ is trivial.

Thus,
\[
\z= \sum_{i=1}^{k}\omega^{\x_i} < \y \leq  \sum_{i=1}^{k+1}\omega^{\x_i}.
\]
As by the definition of $\z$ we see that $\y \leq  \sum_{i=1}^{k+1}\omega^{\x_i}$ cannot be an equality whence
\[
\z= \sum_{i=1}^{k}\omega^{\x_i} < \y <  \sum_{i=1}^{k+1}\omega^{\x_i}.
\]
Thus, $\z\in \mathsf{CNA}(\y)$ and $\y+ \sum_{i=k+1}^{N_{\x}}\omega^{\x_i} = \x$, whence 
\[
\sum_{i=k+1}^{N_{\x}}\omega^{\x_i} = \y{\downarrow}\x=  \sum_{i=1}^{k}\omega^{\x_i} = \z{\downarrow}\x.
\]

2. Follows from part 1 once we realize that for different $\z$ and $\z'$ both in ${\sf CNA}(\x)$ we have $\z{\downarrow}\x\neq \z'{\downarrow}\x$.

3. From the proof of part 1 we see that
\[{\x}{\downarrow}\z = \y{\downarrow}\z \ \Leftrightarrow \ \max \{ \z' {\in} {\sf CNA}(\z) \mid \z'\leq \x \} = \max \{ \z' {\in} {\sf CNA}(\z) \mid \z'\leq \y \}\]
where the latter is precisely the definition of $\x\sim_\z\y$.
\qed

Once we have this lemma to characterize the ${\x}{\downarrow}\y$ operation, we are armed to prove a characterization for when two coordinates in $\vec{\Omega}({A})$ are equal.

\begin{theorem}\label{theorem:EqualityCoordinates}
For $A\in \BNF$, the following five conditions are equivalent.
\begin{enumerate}
\item
$\Omega_{\x}({A}) = \Omega_\zeta({A})$\\
\item
$o_{\x}h_{\x}({A}) = o_\zeta h_\zeta({A})$\\
\item
${\x}{\downarrow} h_{\x}({A}) = \y{\downarrow} h_\zeta({A})$\\
\item
$h_{\x}({A}) = h_\zeta({A})$ and ${\x}{\downarrow} h_{\x}({A}) = \y{\downarrow} h_\zeta({A})$\\
\item
$h_{\x}({A}) = h_\zeta({A})$ and $\forall \z \in h_{\x}({A}), \  \x\sim_\z\y$
\end{enumerate}
\end{theorem}

\proof
$(1.) \Leftrightarrow (2.)$ is just Lemma \ref{theorem:OmegaReducesToO}.

$(2.) \Leftrightarrow (3.)$: Observe that $o_{\x}(h_{\x}({A}))=o({\x}{\downarrow}h_{\x}({A}))$ and $o_\zeta(h_\zeta({A}))=o(\y{\downarrow}h_\zeta({A}))$. As $o$ defines an isomorphism between $\B$ and $\sf On$, we obtain\footnote{As was kindly pointed out by an anonymous referee, it is essential to assume that $A\in \B$. Note that $o_1 h_1 (A) = o_0 h_0(A)$ but $1{\downarrow}h_1(A) \neq 0{\downarrow}h_0(A)$ in case $A=\omega \omega 0 \omega$.}
\[
o_{\x}h_{\x}({A})=o_\zeta h_\zeta({A}) \ \Leftrightarrow \ {\x}{\downarrow}h_{\x}({A}) = \y{\downarrow}h_\zeta({A}).
\] 

$(3.) \Leftrightarrow (4.)$:
Suppose $\xi\downarrow h_\xi(A)=\zeta\downarrow h_\zeta(A)$. Then, it follows that the two have equal length; further, they have length equal to that of $h_\xi(A),h_\zeta(A)$, respectively. But two initial segments of $A$ of equal length must be equal, that is, $h_\xi(A)=h_\zeta(A)$.


$(4.)\Leftrightarrow (5.)$:
\[
\begin{array}{lllll}
h_{\x}({A}) = h_\zeta({A})&\& & {\x}{\downarrow} h_{\x}({A}) = \y{\downarrow} h_\zeta({A}) & \Leftrightarrow & \ \\
h_{\x}({A}) = h_\zeta({A}) &\& & \forall\, \z {\in} h_{\x}({A}) \ {\x}{\downarrow}\z = \y{\downarrow}\z& \Leftrightarrow & \mbox{ by Lemma \ref{theorem:DownArrowCharacterization}.3} \\
h_{\x}({A}) = h_\zeta({A}) &\& &  \forall\, \z {\in} h_{\x}({A}) \ \x\sim_\z\y& \  & \ \\
\end{array}
\]
\qed


\subsection{Limit coordinates}

The results so far have already provided us with quite some insight about what the sequences $\vec{\Omega}({A})$ look like. By Lemma \ref{theorem:WeaklyDecreasingSequences} we know that the set of values that occur in $\vec{\Omega}({A})$ is finite. Moreover, by Theorem \ref{theorem:SuccessorRelations} we know exactly the values at successor coordinates in terms of the direct predecessor. In particular, we know that if the value of $\vec{\Omega}({A})$ at $\x$ is the same as at the successor coordinate, then it remains the same for all further successors. 

The question remains what happens at limit ordinals coordinates. In this subsection we shall determine at what limit ordinals a new value can be attained and how the new value relates to previous values.
Let us start out the analysis by formulating a negative version of Theorem \ref{theorem:EqualityCoordinates}. 

\begin{lemma}\label{theorem:SmallerOmegaCharacterization}
For $A \in \B$ and ${\xi} < {\zeta}$ we have that 
\begin{align*}
\Omega_{{\xi}}({A}) >  \Omega_{\zeta}({A}) &\Leftrightarrow\\
 (\exists \,{\z}{\in} h_{{\xi}}({A})\ {\xi}{\leq} \z {<} {\zeta})& \ \vee \ (\exists \, \z{\in} h_{{\xi}}({A})\ {\sf CNP}({\xi},\z){<} {\sf CNP}({\zeta},\z)).
\end{align*}
\end{lemma}

\proof
By contraposing equivalence $(1.) \Leftrightarrow (5.)$ of Theorem \ref{theorem:EqualityCoordinates} we get 
\[
\Omega_\xi({A}) \neq \Omega_{\zeta}({A}) \ \Leftrightarrow \ h_{\zeta}({A}) {\neq} h_\xi({A}) \ \vee \ \exists \,{\z}{\in}h_\xi({A}) \ {\zeta} \not \sim_{\z} {\xi}.
\]
But, as ${\xi}<{\zeta}$ we see 
\[
h_{\zeta}({A}) {\neq} h_\xi({A}) \ \Leftrightarrow \ \exists\,{\z}{\in}h_\xi({A})\ {\xi}\leq\z <{\zeta}.
\]
Likewise, 
\[
\exists \,{\z}{\in}h_\xi({A}) \ {\zeta} \not \sim_{\z}{{\xi}}\ \Leftrightarrow \ \exists \,{\z}{\in}h_\xi({A}) \ {\sf CNP}({\xi},\z){\neq} {\sf CNP}({\zeta},\z).
\]
As ${\xi}<{\zeta}$ we have 
\[{\sf CNP}({\xi},\z){\neq} {\sf CNP}({\zeta},\z) \ \Leftrightarrow \ {\sf CNP}({\xi},\z){<} {\sf CNP}({\zeta},\z).
\]
\qed

The first question to ask is at which limit coordinates the sequence $\vec{\Omega}({A})$ can change. Let us first write precisely what it means for the sequence $\vec{\Omega}({A})$ to change at some coordinate $\y$. We express this by the expression 
\begin{align*}
{\sf Change}(\y,{A}) \ \  &:=\\
\exists\, \x{<}\y & \ (\Omega_{\x}({A}) {>} \Omega_\zeta({A})\ \& \ \forall \z\ (\x{\leq}\z{<}\y \Rightarrow \Omega_{\x}({A}){=}\Omega_\eta({A}))).
\end{align*}
The next lemma gives an alternative characterization of ${\sf Change}(\y,{A})$.
\begin{lemma}\label{theorem:AlternativeCharacterizationOfChange}
${\sf Change}(\y,{A}) \ \ \Leftrightarrow \ \ \forall\, \x{<}\y \ \Omega_{\x}({A}) {>} \Omega_\zeta({A})$
\end{lemma}

\proof
For $\y \in {\sf Succ}$ this is clear. If $\y\in {\sf Lim}$, then $\{\Omega_{\x}({A})\mid \x<\y\}$ is a finite set as all the $\Omega_{\x}({A})\in {\sf On}$ and these are weakly decreasing. Thus, at some point below $\y$ the sequence must stabilize.
\qed

We can now characterize at what limit ordinals the sequence $\vec{\Omega}({A})$ can change.

\begin{theorem}
For $\y\in {\sf Lim}$ and $A \in \B$: \ \ ${\sf Change}(\y,{A})\ \ \Leftrightarrow \ \ \exists \, \x{\in}h_\zeta({A})\ \y{\in}{\sf CNA}(\x)$.
\end{theorem}

\proof
For $\y\in {\sf Lim}$ we see that, by Lemma \ref{theorem:AlternativeCharacterizationOfChange}, ${\sf Change}(\y,{A})$ is equivalent to the claim that, given $\x<\y$, $\Omega_{\x}({A}) {>} \Omega_\zeta({A})$.

By Lemma \ref{theorem:SmallerOmegaCharacterization}, the latter is in turn equivalent to
\begin{equation}\label{equation:changeEquivalence}
\forall \, \x{<}\y\ (\exists\,{\z}{\in}h_{\x}({A})\ \x{\leq}\z<\y \ \vee \ \exists\,{\z}{\in}h_{\x}({A})\ {\sf CNP}(\x,\z){<}{\sf CNP}(\y,\z)).
\end{equation}
Clearly, if $\exists \, \x{\in}h_\zeta({A})\ \y{\in}{\sf CNA}(\x)$ then for each $\xi' < \zeta$ we have that ${\sf CNP}(\x',\xi){<}{\sf CNP}(\zeta,\xi)$ and by \eqref{equation:changeEquivalence} we conclude ${\sf Change}(\y,{A})$. This concludes the $\Leftarrow$ direction.

For the other direction we reason as follows. Let $\x_0\ :=\ \max \{ \x' \in {A} \mid \x'< \y\} +1$. Since $\zeta$ is a limit, $\x_0 < \zeta$. Clearly, for $\x_0 < \x < \zeta$ we have that $\neg \exists \eta \in h_\x (A) \ \x \leq \eta < \zeta$. Note that for these $\x$'s we have $h_{\x}({A})=h_\zeta({A})$ thus, by  \eqref{equation:changeEquivalence} we obtain
\begin{equation}\label{equation:reducedToCNPs}
\forall \x \ (\x_0{<}\x{<}\y \to \exists\,{\z}{\in}h_\zeta({A})\ {\sf CNP}(\x,\z){<}{\sf CNP}(\y,\z)).
\end{equation}
Now we define $\x_1 < \zeta$ so that it is at least $\x_0$ and exceeds all possible Cantor normal form approximations from ordinals in $h_\zeta (A)$. That is, we define 
\[
\x_1 := \max \{ x_0, \max \{ \lambda \in {\sf CNA}(\eta) \mid \eta \in h_\zeta(A) \wedge \lambda < \zeta \} + 1\}.
\]
Again, $\x_1< \zeta$ since $\zeta$ is a limit. Thus, from \eqref{equation:reducedToCNPs} we in particular get
\[
\forall \x \ (\x_1{<}\x{<}\y \to \exists\,{\z}{\in}h_\zeta({A})\ {\sf CNP}(\x,\z){<}{\sf CNP}(\y,\z)).
\]
However, by the very choice of $\x_1$ for any such $\x$ we can only have $\exists\,{\z}{\in}h_\zeta({A})\ {\sf CNP}(\x,\z)<{\sf CNP}(\y,\z)$ in case $\zeta \in {\sf CNA}(\eta)$ for some $\eta \in h_\zeta({A})$.
\qed

Now that we have fully determined at which limit coordinates a change can occur the only thing left to establish is the size of the change. In other words, if ${\sf Change}(\y,{A})$ for some $\y\in {\sf Lim}$, how does $\Omega_\zeta({A})$ relate to $\Omega_{\x}({A})$ for $\x<\y$?

Here, our functions $\ex^\x$ come back into play:

\begin{theorem}\label{theorem3.6}
Let $\y{\in}{\sf Lim}$ then, for $\theta$ large enough we have that
\[
\Omega_{\theta}({A}) = \ex^{-\theta+\y} \Omega_\zeta({A}) = e_{\le \zeta} \Omega_\zeta(A).
\] 
\end{theorem}

\proof
We pick $\xi$ large enough so that the values of $\Omega_{\x'}({A})$ do not change for $\x\leq \x'<\y$. Thus, we know in particular by Theorem \ref{theorem:EqualityCoordinates} that $h_{\x}({A}) \ =\  h_{\x'}({A})$ whence also
\begin{equation}\label{dagger}
h_{\x'}({A}) \ =\  h_\zeta({A}) \ \ \ \mbox{ for each $\x' \in [\x,\y]$} .
\end{equation}
As $\zeta = \zeta' + \omega^{\le \zeta}$ for some $\zeta' < \zeta$, we have that $-\x + \y \geq \omega^{\le \zeta}$. So certainly $-\theta + \y = \omega^{\le \zeta}$ for $\theta \in [\x,\y)$ large enough. Let $\delta=-\theta+\y = \theta \downarrow \y = \omega^{\le \zeta}$. By definition
\begin{equation}\label{equation:DefinitionOfDelta}
\theta + \delta = \y.
\end{equation}

Now we can prove our theorem: 
\[
\begin{array}{llll}
\Omega_{\theta}({A}) & = & o_{\theta}h_{\theta}({A}) &  \mbox{Lemma \ref{theorem:OmegaReducesToO}}\\
 & = & o_{\theta}h_\y({A})& \mbox{By $(\ref{dagger})$}\\
 & = &o( \theta\downarrow h_\zeta({A}))& \mbox{Lemma \ref{theorem:OrderTypeCalculus}.\ref{down}}\ \\
 & = & o(\delta\uparrow ((\theta + \delta)\downarrow h_\zeta (A))) \ \ \ \ \ \  & \mbox{Lemma \ref{theorem:uparrowProperties}.\ref{inverse}}\\
 & = & o(\delta\uparrow (\zeta\downarrow h_\zeta (A)))& \mbox{By \eqref{equation:DefinitionOfDelta}} \\
 & = & \ex^{\delta} o(\zeta{\downarrow}h_\zeta({A})) & \mbox{Lemma \ref{theorem:OrderTypeCalculus}.\ref{exo}}\\
 & = & \ex^{\delta} o_\zeta h_\zeta({A})) & \mbox{Lemma \ref{theorem:OrderTypeCalculus}.\ref{down}}\\
 & = &e_{\le \zeta}\Omega_\zeta (A). &\mbox{Lemmas \ref{theorem:OrderOFWormsOfOmegaPowersAreFixedPoints} and \ref{theorem:OmegaReducesToO}}
\\
\end{array}
\]
\qed
Note that this theorem establishes the size of limit coordinates both in case a change does occur and in case no change occurs. The latter case can only be so when $\Omega_\zeta({A})$ is a fixed point of $e_{\le\zeta}$.

\section{From local to global}\label{section:FromLocalToGlobal}

The previous section has established exactly where chan\-ges occur in the $\vec{\Omega}(A)$ sequences. Moreover, it established the size of each change in the sequence. We have distinguished two cases: successor coordinates and limit coordinates. 

In Theorem \ref{theorem3.6} we have seen that the value of a limit coordinate fully determines its `direct predecessor' and vice versa. Recall that the value of a successor coordinate is fully determined by the value of its predecessor but not vice versa. Thus, the values of the early coordinates fully determines what comes after it but not so in the other direction.
In Section \ref{section:LinearOrdersOnWorms} we provided a calculus to compute $o(A)$ for given $A$.
Thus, the results in the previous section provide sufficient information to fully calculate $\vec{\Omega}(A)$. 

However, the algorithm implicit in the current results are of a nature that all computations are performed globally: If we wish to compute $\Omega_\zeta(A)$, we need to compute the values of all its predecessors. Thus, first we compute $\Omega_0(A)=o_0(A)$, next we determine at what coordinates the sequence $\vec\Omega(A)$ changes up to $\zeta$. In the end we compute all the successive values of the coordinates where $\vec\Omega(A)$ changes to finally obtain $\Omega_\zeta(A)$.

We shall now see that each change in $\vec\Omega(A)$ is of similar nature so that successively computing the changes corresponds to a certain transfinite iteration. Recall that $\Omega_{\x +1}(A) = \le \Omega_\x (A)$ by Theorem \ref{theorem:SuccessorRelations}. We can see $\le$ as a natural left inverse of $e^1 = e^{-\x + (\x +1)}$ so that
\[
\begin{array}{rcl}
e^1\Omega_{\x +1}(A) &=&  \Omega_\x (A)\\
&\Rightarrow&\\
\le^1 e^1 \Omega_{\x +1}(A) &=&  \le^1\Omega_\x (A)\\
&\Rightarrow&\\
 \Omega_{\x +1}(A) &=& \le\Omega_\x (A).
\end{array}
\] 

If, more generally, for every $\vartheta$ we find an analogous left inverse $\le^\vartheta$ for $\ex^\vartheta$, then we may similarly obtain
\[
\begin{array}{rcl}
e^{-\x+\y} \Omega_{\y}(A) &=& \Omega_\x (A)\\
&\Rightarrow&\\
\le^{-\x+\y} e^{-\x+\y}\Omega_{\y}(A) &=& \le^{-\x+\y}  \Omega_\x\\
&\Rightarrow&\\
\Omega_{\y}(A) &=& \le^{-\x+\y} \Omega_\x (A)
\end{array}
\] 
when $\zeta, \x$ and $A$ are as in Theorem \ref{theorem3.6}. 

In \cite{FernandezJoosten:2012:Hyperations} the authors systematically study natural left-inverses of hyperations and call them \emph{cohyperations}. Once this is in place we can give a global calculus for our sequences, that is, a calculus that computes $\Omega_\zeta(A)$ in `one step' from $\Omega_0(A)$ or from any other previous coordinate.


\section{Hyperations and Cohyperations}\label{section:Cohyperations}

In this section we shall briefly state the main definitions and results from \cite{FernandezJoosten:2012:Hyperations} which are relevant for the current paper. With these at hand we can give a useful characterization of cohyperating the end-exponent function $\le$.

\subsection{Hyperations}
{\em Hyperation} is a form of transfinite iteration of normal functions. It is based on the additivity of finite iterations, that is $f^{m+n} = f^m\circ f^n$ generalizing this to the transfinite setting.

\begin{defi}[Weak hyperation]\label{definition:WeakHyperation}
A {\em weak hyperation} of a normal funcion $f$ is a family of normal functions $\langle g^{\x}\rangle_{\xi\in\mathsf{On}}$ such that
\begin{enumerate}
\item $g^0{\x}={\x}$ for all ${\x}$,
\item $g^1=f$,
\item $g^{{\x}+\zeta}=g^{\x} g^\zeta$.
\end{enumerate}
\end{defi}

\emph{Par abuse de langage} we will often write just $g^{\x}$ instead of $\langle g^{\x}\rangle_{\xi\in\mathsf{On}}$. Weak hyperations are not unique. However, if we impose a minimality condition, we can prove that there is a unique minimal hyperation.

\begin{defi}[Hyperation]
A weak hyperation $g^{\x}$ of $f$ is {\em minimal} if it has the property that, whenever $h^{\x}$ is a weak hyperation of $f$ and ${\x},\zeta$ are ordinals, then $g^{\x}\zeta\leq h^{\x}\zeta$.

If $f$ has a (unique) minimal weak hyperation, we call it {\em the hyperation} of $f$ and denote it $f^{\x}$.
\end{defi}

Hyperations allow for an explicit recursive definition very much in the style of Theorem \ref{theorem:recursiveSchemeForHyperexponentials}. Moreover, there turns out to be a close connection between hyperations and Veblen progressions as shown by the following two theorems.

\begin{theorem}\label{theorem:HyperExponentialsAndVeblen}
Let $f$ be a normal function and let $f_{\alpha}$ be the Veblen progression based on it. Given an ordinal $\alpha$, we have that $f^{\omega^\alpha} =f_\alpha$.
\end{theorem}

\begin{theorem}\label{theorem:VeblenDefinesHyperation}
Let $g^{\xi}$ be a weak hyperation of a normal function $f$. If we moreover have that $g^{\omega^\alpha} =f_\alpha$ for each $\alpha$ then $g^{\xi}=f^\xi$.
\end{theorem}

We will call the functions $e^\alpha$ hyperexponentials. They can be used to define weak normal forms. For example, given an ordinal ${\x}$, we say an expression
\[
{\x}=\sum_{i< I}\ex^{\alpha_i}\beta_i+n
\]
is a {\em Weak Hyperexponential Normal Form} if $I,n<\omega$, and for each $i+1<I$, both $\ex^{\alpha_i} \beta_i \geq \ex^{\alpha_{i+1}}\beta_{i+1}$ and $\beta_i<\ex^{\alpha_i}\beta_i$. Note that Weak Hyperexponential Normal Forms are typically not unique. For example $\omega^\omega = e^21=e^1\omega$. We do, however, have uniqueness if every $\alpha_i$ is of the form $\omega^\delta$.

\begin{lemma}\label{theorem:HyperexponentialNormalForms}
Every ordinal ${\x}>0$ has a weak hyperexponential normal form.

If we further require that every exponent be of the form $\omega^\delta$, then the WHNF obtained is unique.
\end{lemma}

\proof
Write ${\x}$ in Veblen Normal Form and replace $\varphi_\alpha(\beta)$ by $\ex^{\omega^\alpha}(1+\beta)$ for $\alpha >0$, $\varphi_0 (\beta)$ by $e^1(\beta)$ for $\beta>0$. The occurrences of $\varphi_0(0)$ can be captured in the term $+n$ in the end of a WHNF.

If all exponents are of the form $\omega^\delta$, we may invert the process to obtain a VNF from a given WHNF; the uniqueness of the latter follows from the uniqueness of the former.
\qed

\subsection{Cohyperations}

Hyperations are injective and hence invertible on the left; however, a left-inverse of a hyperation is typically not a hyperation, but a different form of transfinite iteration we call {\em cohyperation}. Instead of iterating normal functions we shall consider \emph{initial functions}. We will say a function $f$ is {\em initial} if, whenever $I$ is an initial segment (i.e., of the form $[0,\beta)$ for some $\beta$), then $f(I)$ is an initial segment. It is easy to see that $f\xi \leq \xi$ for initial functions $f$.


\begin{defi}[Cohyperation]
A {\em weak cohyperation} of an initial function $f$ is a family of initial functions $\langle g^{\x}\rangle_{\xi\in\mathsf{On}}$ such that
\begin{enumerate}
\item $g^0{\x}={\x}$ for all ${\x}$,
\item $g^1=f$,
\item $g^{{\x}+\zeta}=g^\zeta g^{\x}$.
\end{enumerate}

If $g$ is maximal in the sense that $g^{\x}\zeta\geq h^{\x}\zeta$ for every weak cohyperation $h$ of $f$ and all ordinals $\xi,\zeta$, we say $g$ is {\em the cohyperation} of $f$ and write $f^{\x}=g^{\x}$.
\end{defi}

Both hyperations and cohyperations are denoted using the superscript; however, this does not lead to a clash in notation as the only function that is both normal and initial is the identity.

There is a general recursive scheme to compute actual cohyperations in the spirit of Definition \ref{theorem:recursiveSchemeForHyperexponentials}. 
\begin{lemma}\label{theorem:RecursionCohyperations}
Every initial function $f$ has a unique cohyperation, given by
\begin{enumerate}
\item $f^0\alpha=\alpha$,
\item $f^1=f$,
\item $f^{\omega^\rho+{\x}}=f^{{\x}}f^{\omega^\rho}$ provided ${\x}<\omega^\rho+{\x}$,
\item $f^{\omega^\rho}{\x}=f^{\omega^\rho}f^\eta {\x}$, if $f^\eta{\x}<{\x}$ and $\eta<\omega^\rho$,
\item $f^{\omega^\rho}{\x}=\sup_{\zeta<{\x}}(f^{\omega^\rho}\zeta+1)$, if $f^\eta{\x}={\x}$ for all $\eta<\omega^\rho$, with $\rho>0$.
\end{enumerate}
\end{lemma}

At first glance it is not even clear that $f^{\x}$ is well defined in that it is single valued. In Item 4., there might be various $\eta$'s below $\omega^{\rho}$ so that $f^{\eta}\x < \x$. In \cite{FernandezJoosten:2012:Hyperations} it is shown that it does not matter which $\eta$ one takes.

Let $f$ be a normal function. Then, $g$ is a {\em left adjoint} for $f$ if, for all ordinals $\alpha,\beta$,
\begin{enumerate}
\item if $\alpha= f(\beta)$, then $g(\alpha)=\beta$ and
\item if $\alpha< f(\beta)$, then $g(\alpha)<\beta$.
\end{enumerate}

Left-adjoints are natural left-inverses and cohyperating them yields left-adjoints to the corresponding hyperations in a uniform way:

\begin{theorem}\label{cancel}
Given a normal function $f$ with left adjoint $g$ and ordinals $\xi<\zeta$ and $\alpha$, $g^\xi f^\zeta=f^{-\xi+\zeta}$ and $g^\zeta f^\x=g^{-\xi+\zeta}$.
\end{theorem}

\begin{theorem}\label{logexp}
The function $\le$ is a left adjoint to $\ex$, and thus $\le^\xi$ is left adjoint to $\ex^\x$ for all $\x$.
\end{theorem}

For the cohyperation of $\le$ we give the following easy recursive scheme.

\begin{theorem}\label{theorem:endLogarithmHyperationRecursiveEquation}
For ordinals $\xi,\zeta$, the value of $\le^\xi\zeta$ is given by the following recursion:
\begin{enumerate}
\item $\le^0\alpha=\alpha$,
\item $\le^\xi n=0$ for $n\in \omega$ and $\xi >0$,
\item $\le^\xi(\alpha+\omega^\beta)=\le^\xi\omega^\beta$ if $\xi >0$,
\item $\le^{\omega^\rho+\xi}=\le^{\xi}\le^{\omega^\rho}$ provided $\xi<\omega^\rho+\xi$,
\item $\le^{\omega^\rho}\ex^{\omega^\beta}\xi=\begin{cases}
\ex^{\omega^\beta}\xi & \mbox{if $\omega^\rho < \omega^\beta$,}\\
\xi & \mbox{if $\omega^\rho= \omega^\beta$,}\\
\le^{\omega^\rho} \xi & \mbox{in case $\omega^\rho > \omega^\beta$.}
\end{cases}$
\end{enumerate}
\end{theorem}

\proof
We shall first see that the recursive scheme of the unique cohyperation of $\le$ as given in Lemma \ref{theorem:RecursionCohyperations} satisfies the recursion of the current theorem. Next, we shall see that the recursion of this theorem has a unique solution. The latter is necessary as we note that it is not fully determined how the last item of the recursion is to be applied, as an ordinal $\zeta$ might be representable as $e^{\omega^\beta} \xi$ in various ways using different $\beta$ and $\xi$. 

That $\le^0\alpha=\alpha$ follows directly from Lemma  \ref{theorem:RecursionCohyperations}. Any $\xi >0$ can be written as $1+\xi'$ so that
\[
\le^{\xi} (\alpha+\omega^\beta) = \le^{1 + \xi'} (\alpha+\omega^\beta) = \le^{\xi'} \le (\alpha+\omega^\beta)= \le^{\xi'} \le \omega^\beta = \le^{1+ \xi'}  \omega^\beta =\le^{\xi}  \omega^\beta.
\]
From this, it directly follows that $\le^\xi n = 0$ for any $\xi>0$ and $n\in \omega$.  Item 4 of the recursion holds trivially. Item 5 follows directly from 
Theorem \ref{logexp} and Theorem \ref{cancel}.

We shall now show unicity. It is clear that we only need to focus on Item 5. 
Thus, we consider  $\le^{\omega^\rho}\ex^{\omega^\beta}\xi$. In \cite{Joosten:2013:VeblenAndTheWorm} and in Proposition 6.3 of \cite{FernandezJoosten:2012:Hyperations} it is shown that  there is a maximal $\alpha$ such that $\ex^{\omega^\beta}\xi = \ex^\alpha \zeta$ for some $\zeta$. We shall prove that $\le^{\omega^\rho}\ex^{\omega^\beta}\xi =\le^{\omega^\rho}\ex^\alpha \zeta$.

Let $\omega^{\alpha_1} + \ldots + \omega^{\alpha_n} =_{\sf CNF} \alpha$ for this particular $\alpha$.  By maximality of $\alpha$, we see that $\beta\leq \alpha_1$. In case $\beta < \alpha_1$ we see by Theorem \ref{theorem:HyperExponentialsAndVeblen} that $e^{\omega^{\alpha}}\zeta$ is a fixpoint of $e^{\omega^{\beta}}$ so that 
\[
\ex^{\omega^\beta}\xi = \ex^\alpha \zeta = e^{\omega^{\beta}} e^{\omega^{\alpha}}\zeta = e^{\omega^{\alpha}}\zeta,
\] 
whence also $\le^{\omega^\rho}\ex^{\omega^\beta}\xi =\le^{\omega^\rho}\ex^\alpha \zeta$.

In case $\beta =\alpha_1$, we see also have that $\le^{\omega^\rho}\ex^{\omega^\beta}\xi =\le^{\omega^\rho}\ex^\alpha \zeta$ as $e^{\omega^{\alpha_1}}$ is injective.
\qed

We will refer to the functions $\le^\xi$ as {\em hyperlogarithms}.

\subsection{Exact sequences}

A nice feature of cohyperations is that, in a sense, they need only be defined locally. To make this precise, we introduce the notion of an {\em exact sequence}.

\begin{defi}
Let $g^{\x}$ be a cohyperation, and $f:\Lambda\to\Theta$ be an ordinal function.

Then, we say $f$ is {\em $g$-exact} if, given ordinals ${\x},\zeta$ with $\x + \zeta < \Lambda$, $f({\x}+\zeta)=g^\zeta f({\x})$.
\end{defi}

A $g$-exact function $f$ describes the values of $g^{\x} f(0)$. However, for $f$ to be $g$-exact, we need only check a fairly weak condition:

\begin{lemma}\label{tfae}
The following are equivalent:
\begin{enumerate}
\item $f$ is $g$-exact
\item for every ordinal ${\x},$ $f({\x})=g^{\x} f(0)$
\item for every ordinal $\zeta>0$ there is ${\x}<\zeta$ such that $f(\zeta)=g^{-{\x}+\zeta}f({\x})$.
\end{enumerate}
\end{lemma}


\section{A global characterization}\label{section:ComputingOmegaCoordinates}

In this section we shall unify the results obtained so far by describing the sequences $\vec\Omega({A})$ using hyperexponentials and -logarithms.

\begin{theorem}\label{istheexact}
Let ${A}$ be a worm.

Then, $\vec\Omega({A})$ is the unique $\le$-exact sequence with $\Omega_0({A})=o({A})$.
\end{theorem}

\proof
In view of Lemma \ref{tfae}, it suffices to show that, given any ordinal $\zeta$, there is ${\x}<\zeta$ such that $\Omega_\zeta({A})=\le^{-{\x}+\zeta}\Omega_{\x}({A})$.

If $\zeta$ is a successor ordinal, write $\zeta={\x}+1$. Then, by Theorem \ref{theorem:SuccessorRelations}, we have that $\Omega_\zeta({A})=\le\Omega_{\x}({A})$.

Meanwhile, if $\zeta$ is a limit ordinal, we know from Lemma \ref{theorem3.6} that, for ${\x}<\zeta$ large enough,
\[
\Omega_{\x}({A}) = \ex^{-\x+\y}\Omega_\zeta({A}).
\]

Applying $\le^{-\x+\zeta}$ on both sides and using Theorem \ref{cancel}, we see that
\[\le^{-\x+\zeta}\Omega_{\x}({A}) = \Omega_\zeta({A}).\]
Thus we can use Lemma \ref{tfae} to see that $\vec\Omega(A)$ is $\le$-exact, so that, for all $\xi$,
\[
\Omega_\xi(A)=\le^\xi\Omega_0(A)=\le^\xi o_0(A),
\]
as claimed.
\qed

Notice by Theorems \ref{istheexact} and \ref{theorem:endLogarithmHyperationRecursiveEquation} that the computations in omega sequences are rather easy if we have written the values in Weak Hyperexponential Normal Form 
(see Lemma \ref{theorem:HyperexponentialNormalForms}) 
and are determined by the last term. 
If, for example, $\Omega_\xi(A) = \alpha + e^{\omega^\zeta}(\beta)$, then the next value where the $\vec{\Omega}(A)$ sequence changes will be in $\xi + \omega^\zeta$ jumping to the new value $\Omega_{\xi + \omega^\zeta}(A)=\beta$.

Further, hyperexponentials give us {\em lower bounds} on $\le$-exact sequences.
The value of $\Omega_\xi(A)$ fully determines the values of $\Omega_\zeta(A)$ for $\zeta > \xi$ but not vice versa. However for $\zeta > \xi$ we do have a lower-bound on $\Omega_\xi(A)$:

\begin{theorem}
Given a worm ${A}$ and ordinals ${\x},\zeta$, $\Omega_{\x}({A})\geq \ex^{\zeta}\Omega_{{\x}+\zeta}({A})$.
\end{theorem}

\proof
Towards a contradiction, assume that there is a worm $A$ and ordinals $\xi<\zeta$ such that $\Omega_\x(A)<\ex^{{-{\x}+\zeta}}\Omega_\zeta(A)$. Then, by Theorem \ref{logexp}, $\le^{-\x+\zeta}\Omega_\x(A)<\Omega_\zeta(A)$.

But this is impossible by Theorem \ref{istheexact}, given that $\le^{-\x+\zeta}\Omega_\x(A)=\Omega_\zeta(A)$.
\qed

\section{Turing progressions revisited}\label{section:TuringRevisited}


In this section we shall interpret our omega sequences in $\glp_\omega$ in terms of Turing progressions. Before doing so, we first need to introduce a slightly generalized notion of Turing progressions where we transfinitely iterate $i$-consistency rather than normal consistency:

\[
\begin{array}{llll}
T_0^i &:=& T;  \\
T_{\alpha +1}^i & :=& T_\alpha^i \cup \{ \la i \ra_{T_\alpha^i} \top\}; & \\
T_\lambda & := & \bigcup_{\alpha < \lambda} T_\alpha & \mbox{for limit $\lambda$.} 
\end{array}
\]

In this section we shall always assume that $\glp_\omega$ is sound w.r.t.\ the base theory $T$. Generalized Turing progressions are not sensitive to adding ``small" elements to the base theory as is expressed by the following lemma. 

\begin{lemma}\label{theorem:GeneralizedTuringProgressionsInvariantToSmallChangesBaseTheory}
For $T$ an elementary presented theory and for any $GLP_\omega$ worm $A$, if $m<n$, then 
\[
(T+mA)_\alpha^n \equiv (T)_\alpha^n +mA \ \ \ \ \ \mbox{for any $\alpha < \epsilon_0$}.
\]
\end{lemma}

\proof
By transfinite induction on $\alpha$. The only interesting case is at successor ordinals.
\[
\begin{array}{llll}
(T+mA)_{\alpha+1}^n & \equiv_{\sf def}& (T+mA)_{\alpha}^n +\la n\ra_{(T+mA)_{\alpha}^n} \top& \ \\
\  & \equiv_{\sf IH} & T_{\alpha}^n +mA +\la n\ra_{T_{\alpha}^n +mA} \top& \\ 
\  & \equiv & T_{\alpha}^n +mA +\la n\ra_{T_{\alpha}^n} (\top\wedge mA)& \mbox{by Lemma \ref{lemma:basicLemma}}\\ 
\  & \equiv & T_{\alpha}^n +mA +\la n\ra_{T_{\alpha}^n} (\top)& \\
\  & \equiv & T_{\alpha+1}^n +mA & \\
\end{array}
\]
\qed

We shall need a generalization of Proposition \ref{theorem:generalizedReductionLemma} which can be found in \cite{Beklemishev:2005:Survey}. In this section, $U \equiv_n V$ will denote that the theories $U$ and $V$ prove exactly the same $\Pi_{n+1}$ sentences.

\begin{lemma}\label{theorem:wormsAndGeneralizedTuringProgressions}
Let $T$ be some elementary presented theory containing $\ea^+$ whose axioms have logical complexity at most $\Pi_{n+1}$ and let $A$ be some worm in $S_n$. We have that 
\[
T+ A \equiv_n T^n_{o_n(A)}.
\]
\end{lemma}

In general we do of course not have that if $U\equiv_{n}V$, then $U + \psi \equiv_n V + \psi$ for theories $U$ and $V$ and formulas $\psi$. However, in the case of Turing progression we can add ``small" additions on both sides and preserve conservativity.

\begin{lemma}\label{theorem:addingSmallWormsToT+A}
Let $T$ be some elementary presented theory containing $\ea^+$ whose axioms have logical complexity at most $\Pi_{n+1}$ and let $A$ be some worm in $S_n$. Moreover, let $B$ be any worm and $m<n$. We have that 
\[
T + A + mB  \equiv_n T^n_{o_n(A)} + mB.
\] 
\end{lemma}

\proof
As $m<n$ we have that $mB \in \Pi_{n}$. Whence, we can apply Lemma \ref{theorem:wormsAndGeneralizedTuringProgressions} to the theory $T+mB$ and obtain
\[
T + mB + A \equiv_n (T+ mB)^n_{o_n(A)}
\] 
However, by Lemma \ref{theorem:GeneralizedTuringProgressionsInvariantToSmallChangesBaseTheory} we see that 
\[
(T+ mB)^n_{o_n(A)}\equiv T^n_{o_n(A)} + mB, \ \mbox{ whence } \ T + mB + A \equiv_n T^n_{o_n(A)} + mB.
\]
\qed
From this lemma we obtain the following simple but very useful corollary.

\begin{cor}\label{theorem:T+AIsAlmostATuringProgression}
Let $T$ be some elementary presented theory containing $\ea^+$ whose axioms have logical complexity at most $\Pi_{n+1}$. Moreover, let $A$ be any worm. We have that 
\[
T + A \equiv_n T^n_{\Omega_n(A)} + r_n(A).
\]
\end{cor}

\proof
We know that $\glp \vdash A \leftrightarrow h_n(A) \wedge r_n(A)$. As by assumption $\glp_\omega$ is sound w.r.t.\ $T$ we see that 
\[
T+ A \equiv T+ h_n(A) + r_n(A).
\]
The worm $r_n(A)$ is either empty or of the form $mA$ for some $m<n$. Clearly, $h_n(A) \in S_n$. Thus, we can apply Lemma \ref{theorem:addingSmallWormsToT+A} and obtain
\[
T+ h_n(A) + r_n(A) \equiv_n T^n_{o_n(h_n(A))} + r_n(A).
\]
However, by Lemma \ref{theorem:OmegaReducesToO} we know that $o_n(h_n(A)) = \Omega_n(A)$ and we are done.
\qed

From Lemma \ref{theorem:wormsAndGeneralizedTuringProgressions} we see that we can capture the $\Pi_1^0$ consequences of the $o(A)$-th Turing Progression of $T$ by the simply axiomatized theory $T+A$. Thus, $T+A$ proves the same $\Pi^0_1$ formulas as $T^0_{o(A)}$. However, $T+A$ will in general prove many new formulas of higher complexity. We can characterize those consequences of $T+A$ also in terms of Turing progressions and the way to do so is simply given by our $\Omega$-sequences.

\begin{theorem}\label{theorem:OmegaSequenceInTuringProgressions}
Let $T$ be some $\Pi_1^0$ axiomatizable elementary representable theory containing $\ea^+$. Let $A$ be any $\glp_\omega$ worm. We have that 
\[
T + A \equiv \bigcup_{i< \omega} T^i_{\Omega_i(A)}.
\]
\end{theorem}

\proof
We prove by induction on $n$ that 
\[
T + A \equiv_n \bigcup_{i=0}^{n} T^i_{\Omega_i(A)}.
\]
This is clearly sufficient as for any $\glp_\omega$ worm $A$ there are only finitely non-zero entries in $\vec{\Omega}(A)$. The base case follows directly from Lemma \ref{theorem:wormsAndGeneralizedTuringProgressions} since $\Omega_0(A) = o_0(A)$.

For the inductive case we reason as follows. By Corollary \ref{theorem:T+AIsAlmostATuringProgression} we know that
\begin{equation}\label{IAmTiredOfTheseLengthyLables}
T + A \equiv_{n+1} T^{n+1}_{\Omega_{n+1}(A)} + r_{n+1}(A).
\end{equation}
In particular, as $T^{n+1}_{\Omega_{n+1}(A)} + r_{n+1}(A) \subseteq \Pi_{n+2}$ we see that actually, $T+A$ is a $\Pi_{n+2}$-conservative extension of $T^{n+1}_{\Omega_{n+1}(A)} + r_{n+1}(A)$, and 
\[
T + A \vdash T^{n+1}_{\Omega_{n+1}(A)} + r_{n+1}(A).
\]
The induction hypothesis tells us that 
\begin{equation}\label{shortLabel}
T + A \equiv_n \bigcup_{i=0}^{n} T^i_{\Omega_i(A)}.
\end{equation}
Again, since $\bigcup_{i=0}^{n} T^i_{\Omega_i(A)} \subseteq \Pi_{n+1}$ we obtain that 
\[
T + A \vdash \bigcup_{i=0}^{n} T^i_{\Omega_i(A)}.
\]
Thus, $T+A \vdash \bigcup_{i=0}^{n+1} T^i_{\Omega_i(A)}$ and in particular, if $\bigcup_{i=0}^{n+1} T^i_{\Omega_i(A)} \vdash \pi$ then $T+A \vdash \pi$ for $\pi \in \Pi_{n+2}$. 

Conversely, assume that $T+A \vdash \pi$ for some $\Pi_{n+2}$ sentence $\pi$. By \eqref{IAmTiredOfTheseLengthyLables} we see that $T^{n+1}_{\Omega_{n+1}(A)} + r_{n+1}(A) \vdash \pi$. However, $r_{n+1}(A) \in \Pi_{n+1}$ and $T+A \vdash r_{n+1}(A)$ so, by \eqref{shortLabel} we see that $\bigcup_{i=0}^{n} T^i_{\Omega_i(A)} \vdash r_{n+1}(A)$. Thus 
\[
\begin{array}{lll}
\bigcup_{i=0}^{n+1} T^i_{\Omega_i(A)} & \vdash & T^{n+1}_{\Omega_{n+1}(A)} + r_{n+1}(A) \\
\ & \vdash & \pi.
\end{array}
\]
as was required.

\qed

In order to obtain a generalization of Theorem \ref{theorem:OmegaSequenceInTuringProgressions} for worms $A$ in $\glp_\Lambda$ for recursive $\Lambda > \omega$ one first would need suitable (hyper)arithmetical interpretations for which $\glp_\Lambda$ is sound and complete. In \cite{FernandezJoosten:2013:OmegaRuleInterpretationGLP} the authors show that such an interpretation exists. A next step would be to establish the necessary conservation properties. However, the modal reasoning for Theorem \ref{theorem:OmegaSequenceInTuringProgressions} entirely carries over to the more general setting of $\glp_\Lambda$.


\bibliographystyle{plain}
\bibliography{References}

\end{document}